\documentclass[11pt,onecolumn,oneside]{amsart}
\usepackage[a4paper, hmargin={2.8cm, 2.8cm}, vmargin={2.5cm, 2.5cm}]{geometry}

\usepackage{graphicx, color}
\usepackage{enumitem}
\usepackage{amssymb}
\usepackage{amsmath}
\usepackage{amsthm}
\usepackage{mathrsfs}
\usepackage{caption, subcaption}
\usepackage{wrapfig}
\usepackage{mathtools}
\usepackage{bm}
\usepackage{advdate}
\usepackage{url}
\usepackage{accents}
\usepackage{hyperref}
\usepackage{cleveref}

\usepackage{tikz}
\usetikzlibrary{positioning,arrows,chains,matrix,scopes,fit,decorations.markings,decorations,decorations.pathreplacing}
\usetikzlibrary{cd}
\usepackage{ytableau}

\usepackage{indentfirst}

	\newcommand{\A}{\ensuremath{\mathbb{A}}}
	
	\newcommand{\Z}{\ensuremath{\mathbb{Z}}}
	\newcommand{\Q}{\ensuremath{\mathbb{Q}}}
	
	\newcommand{\C}{\ensuremath{\mathbb{C}}}

	\newcommand{\Gm}{\ensuremath{\mathbb{G}_m}}

	\newcommand{\Hom}{\ensuremath{\operatorname{Hom}}}
	
	\newcommand{\id}{\ensuremath{\operatorname{id}}}

	\newcommand{\spec}{\ensuremath{\operatorname{Spec}}}

	\newcommand{\im}{\ensuremath{\operatorname{im}}}
		
    \newcommand{\oo}{\ensuremath{\mathcal{O}}}

\newcommand{\sym}{\ensuremath{\operatorname{Sym}}}

\addtocounter{page}{0}
\theoremstyle{plain}
\newtheorem{theorem}{Theorem}[section]
\newtheorem{corollary}{Corollary}[section]
\newtheorem{lemma}{Lemma}[section]
\newtheorem{prop}{Proposition}[section]
\theoremstyle{definition}
\newtheorem{example}{Example}[section]
\newtheorem{definition}{Definition}[section]
\newtheorem{remark}{Remark}[section]

\newtheorem{rec}{Recall}[section]
\newtheorem{set}{Setup}[section]

\title{Virtual K-theoretic invariants of the nested Hilbert scheme on $\C^2$
} 
\author{Felix Minddal }

\begin{document}

\begin{abstract}
    We construct a nested version of the non-commutative Hilbert scheme and embed the nested Hilbert scheme of points on $\C^n$ as the commutativity locus. In the $\C^2$-case, we exhibit this locus as the zero locus of two different sections of bundles and use this description to equip the nested Hilbert scheme of points with a perfect obstruction theory equivalent to that of Gholampour, Sheshmani and Yau. We study the torus equivariant pushforward of the virtual structure sheaf under the map of nested Hilbert schemes forgetting the largest subscheme of the nesting. Using a map of the bundles on the non-commutative Hilbert scheme, we prove that this pushforward is a twist of the virtual structure sheaf on the lower level. Using localization, we show that the twist is by a constant class with values corresponding to the equivariant Euler characteristic of a tautological class of the Hilbert scheme of points. From this, we derive a closed formula for the multivariate generating series of the equivariant virtual Euler characteristic of the nested Hilbert scheme of points. 
\end{abstract}
\maketitle
\tableofcontents
\section{Introduction}
In the theory of virtual classes on the Hilbert scheme of points, the non-commutative Hilbert scheme introduced by Nori \cite{Nori1977Appendix} has played a central role. The description of $\operatorname{Hilb}^d(\A^3)$ as a critical locus of a potential $\operatorname{ncHilb}^{d}(\A^3)$ was used by Behrend, Bryan and Szendrői \cite{behrend2013motivic} in the study of motivic Donaldson-Thomas invariants. Okounkov \cite{Okounkov2017KTheory} used the same description to prove the Nekrasov formula for the generating series of the twisted virtual Euler characteristic on $\operatorname{Hilb}^{d}(\A^3)$.
\begin{theorem}
    (\cite{Okounkov2017KTheory}).
    \[\sum_{d\geq 0}\chi_{T}(\operatorname{Hilb}^{d}(\A^3),\hat{\mathcal{O}}^{\operatorname{vir}})(-q)^d=\operatorname{PE}\left(\frac{[t_1t_2][t_1t_3][t_2t_3]}{[\mathfrak{t}^{\frac{1}{2}}q][\mathfrak{t}^{\frac{1}{2}}q^{-1}][t_1][t_2][t_3]}\right)\]
where $[x]=x^{1/2}-x^{-1/2}$, $\mathfrak{t}=t_1t_2t_3$ and
\[\operatorname{PE}(f(q;t_1,t_2,t_3))=\operatorname{exp}\left(\sum_{n\geq 1}\frac{f(q^n;t_1^n,t_2^n,t_3^n)}{n}\right).\]
\end{theorem}
Though this critical locus description is special to dimension 3, it is not special to the rank 1 case; Beentjes–Ricolfi \cite{non-comap} proves that the Quot scheme $\operatorname{Quot}^{d}_r(\A^3)$ is the critical locus of a potential on a non-commutative Quot scheme $\operatorname{ncQuot}^{d}_{r}(\A^3)$. Using this description, Fasola, Monavari and Ricolfi \cite{Higher-rank-DT} and Arbesfeld and Kononov \cite{ArbesfeldKononovToAppear} proved the formula for the generating series of the twisted virtual Euler characteristic on $\operatorname{Quot}^{d}_r(\A^3)$ conjectured by Awata and Kanno \cite{Awata2009QuiverMM}.
\begin{theorem}
    (\cite{Higher-rank-DT}, \cite{ArbesfeldKononovToAppear}). 
    \[\sum_{d\geq 0}\chi_{T}(\operatorname{Quot}_r^{d}(\A^3),\hat{\mathcal{O}}^{\operatorname{vir}})((-1)^rq)^d=\operatorname{PE}\left(\frac{[\mathfrak{t}^r][t_1t_2][t_1t_3][t_2t_3]}{[\mathfrak{t}][\mathfrak{t}^{\frac{1}{2}}q][\mathfrak{t}^{\frac{1}{2}}q^{-1}][t_1][t_2][t_3]}\right).\]
\end{theorem}
In the $\A^4$-case Kool and Rennemo realize $\operatorname{Quot}^{d}_{r}(\A^4)$ as the zero locus of an isotropic section of an orthogonal bundle on $\operatorname{ncQuot}^{d}_{r}(\A^4)$ setting it in the "standard model" version of the invariants defined by Oh and Thomas \cite{Oh2020CountingSO}. They then consider the virtual Euler characteristic of the twisted structure sheaf tensored with a specific bundle. They identify the generating function with the instanton partition function $\mathbf{Z}^{\operatorname{NP}}_{r}$ of Nekrasov and Piazzalunga \cite{Nekrasov_2019} and prove their conjectured formula.
\begin{theorem}(\cite{kool2025proofmagnificentconjecture}.) 
\[\mathbf{Z}_r^{\operatorname{NP}}=\operatorname{PE}\left(\frac{[t_1t_2][t_1t_3][t_2t_3][y]}{[y^{\frac{1}{2}}q][y^{\frac{1}{2}}q^{-1}][t_1][t_2][t_3][t_4]}\right)\]
where $y=y_1\cdots y_r$ and $t_1t_2t_3t_4=1$.
\end{theorem}
In this paper we continue the scheme of calculating $K$-theoretic virtual invariants using a non-commutative version of the Hilbert scheme, applying it to the nested Hilbert scheme of points on $\A^2$. For a smooth projective surface $S$, integers $n_0\geq n_2\geq \cdots \geq n_N\geq 0$ and cohomology classes $\beta_0,\dots, \beta_{N-1}\in H^2(S,\Z)$ Gholampour, Sheshmani and Yau \cite{GHOLAMPOUR2020107046}define perfect obstruction theory on the nested Hilbert scheme $S^{[\mathbf{n}]}_{\mathbf{\beta}}$ parameterizing 
\[(Z_0,\dots, Z_N), \; (C_0,\dots, C_{r-1})\]
where $Z_i$ is a zero-dimensional subscheme of length $n_i$ and $C_i$ is a divisor with $[C_i]=\beta_i$ such that
\[I_{Z_i}(-C_i)\subset I_{Z_{i+1}}\]
for all $i=0,\dots, N-1$. This includes the case of the nested Hilbert scheme of points for $\beta_0=\beta_2=\cdots =\beta_{N-1}=0$. Gholampour and Thomas \cite{Gholampour2017DegeneracyLV} exhibit the nested Hilbert scheme of points on $S$ as a degeneracy locus of $S^{[n_0]}\times \cdots S^{[n_N]}$. From this, they equip it with a perfect obstruction theory equivalent to that of \cite{GHOLAMPOUR2020107046}. They further use a Thom-Porteous formula to express the pushforward of the virtual fundamental class to $S^{[n_0]}\times \cdots S^{[n_N]}$ in terms of Carlsson-Okounkov operators in the sense of \cite{Carlsson_2012}. For integers $r\geq 1$ and $n_0\geq n_2\geq \cdots \geq n_N\geq n_{N+1}=0$ Bonelli, Fasola and Tanzini \cite{Bonelli_Fasola_Tanzini_2024} define the moduli space $\mathcal{N}(r,\underline{n})$ of stable representations of the nested instantons quiver of numerical type $(r,\mathbf{n})$. They prove that $\mathcal{N}(1,\mathbf{n})\cong (\A^2)^{[\hat{\mathbf{n}}]}$ where $\hat{n}_i=n_0-n_{N+1-i}$. They further equip $\mathcal{N}(r,\underline{n})$ with a perfect obstruction theory and prove for $r=1$ that this is equivalent to that of \cite{GHOLAMPOUR2020107046}.\newline

In this paper, we consider the nested Hilbert scheme of points $\operatorname{NHilb}^{(d_0,\dots, d_N)}(\A^2)=S^{[\hat{\mathbf{n}}]}\cong \mathcal{N}(1,\mathbf{n})$ where $d_i=n_i-n_{i+1}$. We exhibit it as the zero locus of a section of a bundle on a non-commutative version $\operatorname{ncNHilb}^{(d_0,\dots, d_N)}(\A^2)$ and prove that the induced perfect obstruction theory agrees with that of \cite{GHOLAMPOUR2020107046}, \cite{Gholampour2017DegeneracyLV} and \cite{Bonelli_Fasola_Tanzini_2024}. We will use this to compute pushforwards of virtual structure sheaves allowing us to prove the following formula for the generating series. 
\begin{theorem}
    \[\sum_{d_0,\dots, d_N\geq 0}q_0^{d_0}\cdots q_N^{d_N}\chi_{T}(\operatorname{NHilb}^{(d_0,\dots ,d_N)}(\A^2),\mathcal{O}^{\operatorname{vir}})=\operatorname{PE}\left(\frac{q_0+(q_1+\cdots+q_N)(1-\mathfrak{t})}{(1-t_1)(1-t_2)}\right)\]
    where $\mathfrak{t}=t_1t_2$.
\end{theorem}
\subsection{Acknowledgments} I thank my advisor Gergely Bérczi for years of valuable discussions forming my geometric intuition of the Hilbert scheme of points. I also thank Sergej Monavari for many insightful discussions on the virtual theory and for helpful talks about the proof of the formula. I also thank Richard Thomas and Martijn Kool for answering my questions around the topic. 
\section{The results}
For any $r\geq 0$ and any $\underline{d}=(d_0,\dots, d_r)\in \Z_{\geq 0}^r$ we construct the nested non-commutative Hilbert scheme 
\[\operatorname{ncNHilb}^{\underline{d}}(\A^2)=\{(N_\bullet, A_1,\dots, A_n,v)\; |\; A_i\in \operatorname{End}^{\operatorname{fil}}_{\C}(N_\bullet), \,v\in N_0,\, \C\langle A_1,\dots, A_n\rangle v=N_0 \}\big{/}\sim\]
which parameterizes flag-preserving operators $A_i:N_\bullet\to N_\bullet$ with a cyclic vector $v\in N_\bullet$ for a flag of vector spaces
\[N_\bullet=N_0\supset \cdots \supset N_{r+1}=0\]
with $\operatorname{dim}_{\C}N_i/N_{i+1}=d_i$. We construct an embedding from the nested Hilbert scheme of points on $\C^n$
\begin{align*}\operatorname{NHilb}^{\underline{d}}(\A^n)&=\{Z_{1}\subset \cdots \subset Z_{r+1}\subset \C^n \;  |\; \operatorname{dim}Z_i=0, \, \operatorname{length}(Z_i)=d_0+\cdots +d_{i-1}\}\\
&=\{I_{r+1}\subset \cdots \subset I_1\subset I_0=\C[x_1,\dots,x_n] \; |\; \operatorname{dim}_{\C}I_i/I_{i+1}=d_i\}
\end{align*}
which maps a nesting of ideals to the flag of vector spaces $N_\bullet$ with $N_i=I_i/I_{r+1}$ with operators given by left multiplication with $x_j$ and the cyclic vector given by the image of the unit $\overline{1}\in I_0/I_{r+1}$. We identify the image of this embedding with the locus of $\operatorname{ncNHilb}^{\underline{d}}(\A^2)$ where the operators $A_1,\dots, A_n$ commute. In the case where $n=2$ this locus is given by the vanishing of 
\[[A_1,A_2]\in \operatorname{End}^{\operatorname{fil}}_{\C}(N_\bullet)\]
and we prove using cyclicity that the vanishing of
\[\overline{[A_1,A_2]}\in \operatorname{Hom}_{\C}^{\operatorname{fil}}(N_{\bullet},N_{\bullet}/\langle v\rangle)\]
is sufficient. This allows us to exhibit the nested Hilbert scheme of points on $\C^2$ as the zero locus of a section of two different tautological bundles, thus equipping it with perfect obstruction theories. For $d_0\neq0$ we prove that the second bundle respectively first bundle gives a perfect obstruction theory with the $K$-theory of the virtual tangent bundle equal to the virtual tangent bundles of 
\[\operatorname{NHilb}^{(d_0,\dots, d_r)}(\A^2), \text{ respectively }\operatorname{NHilb}^{(0,d_0,\dots ,d_r)}(\A^2)\]
equipped with the perfect obstruction theories of Gholampour, Sheshmani and Yau \cite{GHOLAMPOUR2020107046}. In particular, the virtual invariants agree with that of \cite{GHOLAMPOUR2020107046}. We show that the morphism
\[p:\operatorname{NHilb}^{(d_0,\dots, d_r)}(\A^2)\to \operatorname{NHilb}^{(d_0,\dots, d_{r-1})}(\A^2) \]
extends to the non-commutative nested Hilbert scheme. Letting $\mathcal{E}_{\underline{d}}$ be one of the bundles above we construct a surjection 
\[\mathcal{E}_{(d_0,\dots ,d_r)}\twoheadrightarrow p^*\mathcal{E}_{(d_0,\dots, d_{r-1})}\]
which in turn gives us a factorization
\[\oo_{(d_0,\dots, d_r)}^{\operatorname{vir}}=p_*(\mathcal{F})\cdot \mathcal{O}_{(d_0,\dots,d_{r-1})}^{\operatorname{vir}}\]
in $K$-theory for a certain explicit tautological class $\mathcal{F}$. Using localization, we then calculate $p_*\mathcal{F}$ and show that it is a constant class depending only on $d_r$. Explicitly, we have the following main result.
\begin{theorem} \label{Main theorem}
    For $r\in \Z_{\geq 0}$ let $\underline{d}=(d_0,\dots ,d_{r})\in \Z_{\geq 0}^{r+1}$ be a dimension vector and let $\underline{\hat{d}}=(d_0,\dots, d_{r-1})\in \Z_{\geq 0}^{r}$. Let 
    \[p_{\underline{d},\underline{\hat{d}}}: \operatorname{NHilb}^{\underline{d}}(\A^2)\to \operatorname{NHilb}^{\underline{\hat{d}}}(\A^2)\]
    and 
    \[q_{\underline{\hat{d}}}:\operatorname{NHilb}^{\underline{\hat{d}}}(\A^2)\to \operatorname{pt}\]
    be the projections. Let $\oo^{\operatorname{vir}}_{\underline{d}},\oo_{\underline{\hat{d}}}^{\operatorname{vir}}$ be the virtual structure sheaves from the perfect obstruction theory of \cref{POT DEF} or equivalently \cite{GHOLAMPOUR2020107046}. We have an identity
    \[(p_{\underline{d},\underline{\hat{d}}})_*\mathcal{O}^{\operatorname{vir}}_{\underline{d}}=(q_{\underline{\hat{d}}}^*\alpha_{d_r})\cdot \oo_{\underline{\hat{d}}}^{\mathcal{\operatorname{vir}}}\]
    where 
    \[\alpha_{d_r}=\chi_{T}(\operatorname{NHilb}^{d_r}(\A^2),\wedge_{-1}^{\bullet}\left(\omega_{\A^2}^{[d_r]}\right))\]
    for $\omega_{\A^2}=\Omega_{\C^2}^2$ the dualizing sheaf on $\C^2$ and $\omega_{\A^2}^{[d_r]}$ the corresponding tautological bundle. 
\end{theorem}
Note that in the non-nested case we have $\oo_{(d)}^{\operatorname{vir}}=\mathcal{O}_{\operatorname{Hilb}^{d}(\A^n)}$, i.e. the virtual structure sheaf is the structure sheaf. Using this and the known closed formulas for the $T$-equivariant Euler characteristic of $\mathcal{O}_{\operatorname{Hilb}^{d}(\A^n)}$ and $\wedge_{-1}^{\bullet}(\omega_{\A^2}^{[d]})$ the above then immediately implies a closed formula for the multivariate generating series of the equivariant virtual holomorphic Euler characteristic of the nested Hilbert scheme of points on $\C^2$.
\begin{corollary}
For any $r\geq 0$ let 
\[\mathbf{Z}_{\operatorname{GSY}}(q_0,\dots,q_r;t_1,t_2):=\sum_{d_0,\dots, d_r}q_0^{d_0}\cdots q_{r}^{d_r}\chi_T(\operatorname{NHilb}^{(d_0,\dots,d_r)}(\A^2),\oo^{\operatorname{vir}}_{(d_0,\dots,d_r)})\]
be the generating function for the $T$-equivariant virtual holomorphic Euler characteristic of the nested Hilbert scheme with perfect obstruction theory equivalent to that of Gholampour, Sheshmani and Yau \cite{GHOLAMPOUR2020107046}. Then
\[Z_{\operatorname{GSY}}(q_0,\dots,q_r;t_1,t_2)=\operatorname{PE}\left(\frac{q_0+(q_1+\cdots+q_r)(1-t_1t_2)}{(1-t_1)(1-t_2)}\right).\]
\end{corollary}
\begin{proof}
Set 
\[\beta_{d_0}=\chi_{T}(\operatorname{Hilb}^{d_0}(\A^2),\oo^{\operatorname{vir}})=\chi_{T}(\operatorname{Hilb}^{d_0}(\A^2),\oo).\]
The above then implies that
\[\chi_T(\operatorname{NHilb}^{(d_0,\dots,d_r)}(\A^2),\oo^{\operatorname{vir}}_{(d_0,\dots,d_r)})=\beta_{d_0}\alpha_{d_1}\alpha_{d_2}\cdots \alpha_{d_r}.\]
We note that the classic known formulas show that
\[\alpha_{d_i}=[x^{d_i}]\operatorname{PE}\left(\frac{x(1-t_1t_2)}{(1-t_1)(1-t_2)}\right)\]
while
\[\beta_{d_1}=[x^{d_1}]\operatorname{PE}\left(\frac{x}{(1-t_1)(1-t_2)}\right).\]
To see this we note that $\omega_{\A^2}=t_1t_2\mathcal{O}_{\A^2}$. The results then follow from \cite[Theorem 3.2]{Wang2014-oz} with respectively $u=1,v=0$ and $A=(1,1)$ and $u=v=0$ and $A=(0,0)$. We now conclude 
\begin{align*}
    Z_{\operatorname{GSY}}(q_0,\dots,q_r;t_1,t_2)&=\left(\sum_{d_0\geq 0}q_0^{d_0}\beta_{d_0}\right)\left(\sum_{d_1\geq 0}q_1^{d_1}\alpha_{d_1}\right)\cdots \left(\sum_{d_r\geq 0}q_r^{d_r}\alpha_{d_r}\right)\\
    &=\operatorname{PE}\left(\frac{q_0}{(1-t_1)(1-t_2)}\right)\prod_{i=1}^r\operatorname{PE}\left(\frac{q_i(1-t_1t_2)}{(1-t_1)(1-t_2)}\right)\\
&=\operatorname{PE}\left(\frac{q_0+(q_1+\cdots+q_r)(1-t_1t_2)}{(1-t_1)(1-t_2)}\right).
\end{align*}
\end{proof}
\section{Preliminaries}
\subsection{Notation and conventions}
For an algebraic group $G$ and a $G$-scheme $X$ we will denote by $K_0^G(X)$ the $K$-theory of $G$-equivariant coherent sheaves and by $K_G^0(X)$ the $K$-theory of $G$-equivaraint locally free sheaves. \newline 

\noindent We will use the convention that $T=\Gm^n$ acts on $\A^n$ by 
\[(t_1,\dots, t_n).(p_1,\dots, p_n)=(t_1^{-1}p_1,\dots, t_n^{-1}p_n)\]
so that $\mathcal{O}_{\A^n}=\C[x_1,\dots, x_n]$ is the $T$-representation where $t.x_i=t_ix_i$. \newline 

\noindent We will keep track of the number of points of the nested Hilbert scheme of points by keeping track of increments. The subschemes will be indexed as ascending chains starting at 1, while the increments will be indexed starting at 0. In other words, the nested Hilbert scheme $\operatorname{NHilb}^{(d_0,\dots, d_r)}(\A^n)$ of $(d_0,\dots, d_r)$ points on $\A^n$ will parameterize finite subschemes
\[Z_{r+1}\subset \cdots \subset Z_1\subset \A^n\]
such that $\operatorname{length}Z_1=d_0$ and 
\[\operatorname{length}Z_{i+1}-\operatorname{length}Z_i=d_i.\]
If we let $I_i$ be the ideal of $Z_i$ with $I_0=\mathcal{O}_{\A^n}$ we are equivalently parameterizing 
\[I_{r+1}\subset \cdots \subset I_{1}\subset I_0\]
such that $\operatorname{dim}_{\C}I_i/I_{i+1}=d_i$.
\subsection{Perfect obstruction theories and zero loci}
Recall that for a scheme $X$ a perfect obstruction theory as defined in \cite{BehrendFantechi} is a map 
\[\phi:\mathbb{E}\to \mathbb{L}_X\]
in the derived category $D^{[-1,0]}(X)$ such that $h^0(\phi)$ is an isomorphism and $h^{-1}(\phi)$ is surjective. Here, $\mathbb{L}_X=\tau_{-1\geq }L_X^{\bullet}$ is the truncation of the cotangent complex of $X$ as defined in \cite{Illusie1971}. If
\[\mathbb{E}=[E^{-1}\to E^{0}]\]
then the perfect obstruction theory gives rise to a cone
\[\mathfrak{C}\hookrightarrow E_1=(E^{-1})^{\vee}\]
which in turn gives rise to the \textit{virtual fundamental class}
\[[X]^{\operatorname{vir}}=i^*[\mathfrak{C}]\in A_{\operatorname{vd}}(X)\]
where $i:X\to E_1$ is the zero section and $i^*$ its Gysin map. Here, 
\[\operatorname{vd}=\operatorname{rk}\mathbb{E}=\operatorname{rk}E^0-\operatorname{rk}E^{-1}\]
is the virtual dimension and $A_*(X)$ is the Chow group of $X$. Similarly, we can define a \textit{virtual structure sheaf}
\[[Li^*\oo_{\mathfrak{C}}]\in K_0(X)\]
for $K_0(X)$ the $K$-theory of coherent sheaves on $X$. Suppose $X$ is proper. For any Chow cohomology class $\alpha\in A^*(X)$ and any class $V\in K^0(X)$ in the $K$-theory of locally free sheaves we may define virtual invariants by
\[\int_{[X]^{\operatorname{vir}}}\alpha\in A_*(\operatorname{pt})\cong\Z\]
and 
\[\chi^{\operatorname{vir}}(X,V):=\chi(X,V\otimes_X \oo_X^{\operatorname{vir}})\in K_0(\operatorname{pt})\cong \Z.\]
By \cite{Fantechi2010-va} these are related by the virtual Hirzebruch-Riemann-Roch formula
\[\chi^{\operatorname{vir}}(X,V)=\int_{[X]^{\operatorname{vir}}}\operatorname{ch}(V)\cdot \operatorname{Td}(T_X^{\operatorname{vir}})\]
where 
\[T_{X}^{\operatorname{vir}}=[\mathbb{E}^{\vee}]=[E_0]-[E_1]\in K^0(X)\]
is the virtual tangent bundle. Both the virtual fundamental class (see \cite[Theorem 4.6]{Siebert2004}) and the virtual structure sheaf (see \cite[Corollary 4.5]{Thomas2022-la}) only depend on the $K$-theory class $[\mathbb{E}]$.
\begin{example} \label{simple zero locus}
    We are primarily interested in the following simpler situation. Suppose that there exists a smooth scheme $Y$ with a locally free sheaf $E$ equipped with a section $s\in H^0(Y,E)$ such that 
    \[X=Z(S)\subset Y.\]
    Let $I\subset \mathcal{O}_Y$ be the ideal sheaf of $X$ in $Y$. Then the truncated cotangent bundle of $X$ can be expressed as 
    \[\mathbb{L}_X=\left[I/I^2\xrightarrow{d} \Omega_{Y}|_X\right].\]
    The section defines a map 
    \[s^\vee: E^\vee\to \mathcal{O}_Y\]
    with image $I$ so that 
    \begin{center}
        \begin{tikzcd}
            \mathbb{E}\;=\;\big{[}E^\vee|_X \arrow[shift left=15pt]{d}{s^\vee} \arrow{r}{d\circ s^\vee} & \Omega_Y|_X \big{]} \arrow[equal]{d} \\
            \mathbb{L}_X=\big{[} I/I^2 \arrow{r}{d} & \Omega_Y|_X \big{]}
        \end{tikzcd}
    \end{center}
    defines a perfect obstruction theory. In this case the virtual invariants are well known. The virtual fundamental class is simply Fultons refined Euler class \cite[Section 14.1]{Fult}, that is 
    \[[X]^{\operatorname{vir}}=0^!([Y])\]
    where $0:Y\to E$ is the zero section and $0^!$ is the refined Gysin map. For the virtual structure sheaf we see that the $K$-theory class of the Koszul complex $\operatorname{Kos}^{\bullet}(E^\vee,s^\vee)$ associated to $(E^\vee,s^\vee)$
    \[\wedge^{\bullet}_{-1}[E^{\vee}]=\sum_{i\geq 0}(-1)^i[\wedge^iE^{\vee}]=\sum_{i\geq 0}(-1)^i[\mathcal{H}^i(\operatorname{Kos}(E^\vee,s^\vee))]\in K_0(Y)\]
    is supported on $X$ and defines the virtual structure sheaf $\mathcal{O}_X^{\operatorname{vir}}\in K_0(X)$ (this is explained in \cite[Section 3.2]{Okounkov2017KTheory}).
\end{example}
\subsection{Torus actions and localization}
We fix a split torus $T=\Gm^n$. We let $t_1,\dots, t_n\in K^0_T(\operatorname{pt})$ be the classes corresponding to the one-dimensional representations of the characters given by projection to the $i$'th coordinate. Let $s_i=c_1(t_i)$. Then 
\[K_0^T(\operatorname{pt})\cong K^0_T(\operatorname{pt})\cong\Z[t_1^{\pm},\dots, t_n^{\pm}]\]
and 
\[A_*^T(\operatorname{pt})\cong A^*_T(\operatorname{pt})\cong \Z[s_1,\dots, s_n].\]
Let $X$ be a scheme with a $T$-action. A $T$-equivariant perfect obstruction theory is a map 
\[\mathbb{E}\to \mathbb{L}_X\]
in $D^{[-1,0]}(\operatorname{Coh}^T(X))$ such that the underlying morphism in $D^{[-1,0]}(X)$ is a perfect obstruction theory \cite{GraberPandharipande1999}. Given such, the virtual fundamental class and structure sheaf admits equivariant refinements, i.e. 
\[[X]^{\operatorname{vir}}\in A_{\operatorname{vd}}^{T}(X)\]
and
\[\mathcal{O}_{X}^{\operatorname{vir}}\in K_0^T(X).\]
In particular, if $X$ is proper and $\alpha\in A_T^i(X)$ and $V\in K_T^0(X)$ we get equivariant virtual invariants
\[\int_{[X]^{\operatorname{vir}}}\alpha\in A_*^{T}(\operatorname{pt})\cong \Z[s_1,\dots, s_n]\]
and 
\[\chi_T^{\operatorname{vir}}(X,V)=\chi_T(X,V\otimes_X\mathcal{O}_X^{\operatorname{vir}})\in K_0^T(\operatorname{pt})\cong \Z[t_1,\dots, t_n]\]
which are again related by the virtual Hirzebruch-Riemann-Roch formula
\[\operatorname{ch}_T(\chi_T^{\operatorname{vir}}(X,V))=\int_{[X]^{\operatorname{vir}}}\operatorname{ch}_T(V)\cdot \operatorname{Td}_T(T_X^{\operatorname{vir}}).\]
\begin{example}
Suppose $X=Z(s)\subset Y$ for $Y$ a smooth scheme with a $T$-action and $s\in H^0_T(Y,E)$ a section of a global $T$-equivariant section of an equivariant sheaf $E$. Then \cref{simple zero locus} naturally produces a $T$-equivariant perfect obsturction theory. We again have 
    \[[X]^{\operatorname{vir}}=0^![Y]\]
    for $0^!$ the $T$-equivariant refined Gysin map and 
    \[\mathcal{O}^{\operatorname{vir}}=[\wedge^{\bullet}_{-1}E^{\vee}]=[\operatorname{Kos}^{\bullet}(E^\vee,s^{\vee})]\]
    where the Koszul complex is now of $T$-equivariant sheaves.
\end{example}
The main way to calculate these classes is through virtual localization. The general formula of Thomason \cite{Thomason1992-rm} states that for $\iota:X^T\hookrightarrow X$ the inclusion of the $T$-fixed locus the map
\[\iota_*:K_0^T(X^T)\to K_0^T(X)\]
becomes an isomorphism after applying 
\[-\otimes_{K^0_T(\operatorname{pt})}K_T^0(\operatorname{pt})_{\operatorname{loc}}\]
where 
\[K_T^0(\operatorname{pt})_{\operatorname{loc}}=K_T^0(\operatorname{pt})\left[\left.\frac{1}{1-t^a}\; \right| \; a\in \Z^{n}\right]\]
i.e. after we localize in the coefficients $1-t^a=1-t_1^{a_1}\cdots t_n^{a_n}$. 
\begin{definition}
    When $X^T$ is finite and reduced it follows that 
    \[K_0^T(X)_{\operatorname{loc}}=\bigoplus_{x\in X^T}K_0^{T}(\operatorname{pt})_{\operatorname{loc}}.\]
    For $V\in \operatorname{K}_0^T(X)$ we will denote by 
    \[\operatorname{cont}_x(V)\in K_0^T(\operatorname{pt})_{\operatorname{loc}}\] 
    the corresponding unique class corresponding to $x\in X^T$ such that
    \[\iota_*\sum_{x\in X^T} \operatorname{cont}_x(V)=V.\]
    We will refer to it as the contribution of $V$ at $x\in X$.
\end{definition}
In the case of the virtual fundamental class, we have a formula for the unique class. When we restrict the perfect obstruction theory to the fixed locus, it splits into a fixed and a moving part 
\[\iota^*\mathbb{E}=\mathbb{E}^{\operatorname{fix}}\oplus\mathbb{E}^{\operatorname{mov}}\]
corresponding to the summands with trivial respectively non-trivial weights in the eigensheaf decomposition. The induced map  
\[\iota^*\phi^{\operatorname{fix}}: \mathbb{E}^{\operatorname{fix}}\to (i_*\mathbb{L}_{X})^{\operatorname{fix}}\to \mathbb{L}_{X^T}\]
is a perfect obstruction thus inducing a virtual structure sheaf $\mathcal{O}_{X^T}^{\operatorname{vir}}$. If we define the virtual normal bundle as $N^{\operatorname{vir}}=(\iota^*T_X^{\operatorname{vir}})^{\operatorname{mov}}$ where $T_{X}^{\operatorname{vir}}=\mathbb{E}^\vee$ is the virtual tangent bundle, then \cite[Theorem 3.3]{Qu2018-uo} states that
\[\mathcal{O}_{X}^{\operatorname{vir}}=\iota_*\left(\frac{\mathcal{O}_{X^T}^{\operatorname{vir}}}{\wedge_{-1}^{\bullet}[N^{\operatorname{vir,\vee}}]}\right).\]
In particular, this allows us to define the virtual Euler characteristic for non-proper $X$.
\begin{definition}
    Let $X$ be a scheme with a $T$-action and a $T$-equivariant perfect obstruction theory. Suppose $X^T$ is proper. For $V\in K^0_T(X)$ we define 
    \[\chi_T^{\operatorname{vir}}(X,V):=\chi_T^{\operatorname{vir}}\left(X^T,\frac{V|_{X^T}}{\wedge_{-1}^{\bullet}[N^{\operatorname{vir,\vee}}]}\right)\in K_0^T(\operatorname{pt})_{\operatorname{loc}}.\]
    By the above, it agrees with the usual virtual Euler characteristic whenever $X$ is proper. 
\end{definition}
\begin{example}
    Suppose $X^T$ is finite and reduced. Note that for a sum of one-dimensional representations $V=\sum\alpha_i -\sum\beta_j\in K_T^0(\operatorname{pt})$ we have that
    \[\wedge_{-1}^{\bullet}V=\frac{\prod(1-\alpha_i)}{\prod(1-\beta_j)}.\]
    We further note that $\mathbb{L}_{\operatorname{pt}}=0$ so that an equivariant perfect obstruction theory is simply given by a $T$-representation $E$ with $\mathbb{E}=E[-1]$. It follows that 
    \[\mathcal{O}_{\operatorname{pt}}^{\operatorname{vir}}=\wedge_{-1}^{\bullet}[E].\]
    If $\mathbb{E}$ has only trivial weights (as will be the case for $\mathbb{E}_{x}^{\operatorname{fix}})$ we get
    \[\oo_{\operatorname{pt}}^{\operatorname{vir}}=\left\{\begin{array}{cl}
        \mathcal{O}_{\operatorname{pt}} &\text{if } \mathbb{E}=0\\
        0 & \text{else}
    \end{array}\right.\]
    since this is true for $\wedge^{\bullet}_{-1}[E]$. We conclude that the localization formula in this case is given by
    \[\mathcal{O}_{X}^{\operatorname{vir}}=\iota_*\sum_{\underset{\mathbb{E}^{\operatorname{fix}}_{x}=0}{x\in X^T}}\frac{[\mathcal{O}_x]}{\wedge_{-1}^{\bullet}(T_{X,x}^{\operatorname{vir},\vee})}\]
    and in particular writing $T_{X,x}^{\operatorname{vir}}=\sum \alpha_{x,i}-\sum\beta_{x,j}$ we get
    \[\chi_T^{\operatorname{vir}}(X,V)=\sum_{\underset{\mathbb{E}^{\operatorname{fix}}_{x}=0}{x\in X^T}}\frac{V_x\prod(1-\overline{\beta}_{x,j})}{\prod(1-\overline{\alpha}_{x,i})}\]
    where $\overline{(-)}:K_T^0(\operatorname{pt)}\to K_T^0(\operatorname{pt})$ is given by linearly extending $t_i^{\pm}\mapsto t_i^{\mp}$.
\end{example}
\subsection{Relative perfect obstruction theories}
Let $f:X\to Y$ be a map of schemes. A relative perfect obstruction theory is a map 
\[\phi_f:\mathbb{E}_f\to \mathbb{L}_f\]
in $D^{[-1,0]}(X)$ satisfying that $h^0(\phi)$ is an isomorphism and that $h^{-1}(\phi)$ is surjective. A compatible triple is a commuting diagram 
\begin{center}
    \begin{tikzcd}
        f^*\mathbb{E}_{Y} \arrow{d}{f^*\phi_Y} \arrow{r} & \mathbb{E}_X\arrow{d}{\phi_X} \arrow{r} & \mathbb{E}_f \arrow{d}{\phi_f} \\
        f^*\mathbb{L}_Y \arrow{r} & \mathbb{L}_X \arrow{r} & \mathbb{L}_f
    \end{tikzcd}
\end{center}
where the rows are distinguished triangles with the lower row the fundamental triangle and such that $\phi_X,\phi_Y$ and $\phi_f$ are perfect obstruction theories. Given a compatible triple \cite{virtpull} shows that we get a virtual pullback 
\[f^!:A_*(Y)\to A_*(X)\]
such that $f^![Y]^{\operatorname{vir}}=[X]^{\operatorname{vir}}$. Similarly, \cite{Qu2018-uo} constructs a virtual pullback
\[f^!:K_0(Y)\to K_0(X)\]
such that $f^!\mathcal{O}_{Y}^{\operatorname{vir}}\to \mathcal{O}_X^{\operatorname{vir}}$. Both of these pullbacks are bivariant classes in the sense of \cite{Fult}. We briefly show a way to construct a compatible relative perfect obstruction theory when we have two different zero loci descriptions.
\begin{lemma} \label{relative POT criterion}
    For $i=1,2$ let $Y_i$ be a smooth scheme with a locally free sheaf $E_i\in \operatorname{Coh}(Y_i)$ and a section $s_i\in H^0(Y_i,E_i)$. Let $Z_i=Z(s_i)$ equipped with the perfect obstruction theory $\phi_i:\mathbb{E}_i\to \mathbb{L}_{Z_i}$ coming from being exhibited as a zero locus of a section of a locally free sheaf. Suppose we have a morphism of schemes 
    \[f:Y_1\to Y_2\]
    and a morphism of sheaves 
    \[\psi: E_1\to f^*E_2\]
    such that $\psi(s_1)=f^*s_2$ satisfying the transversality condition that the map of sheaves on $Z_1$
    \[f^*(E_2^{\vee}|_{Z_2})\xrightarrow{\left(\begin{array}{c}
        -\text{d}s_2^{\vee}   \\
         \psi^{\vee}
    \end{array}\right)} f^*(\Omega_{Y_2}|_{Z_2})\oplus E_1^{\vee}|_{Z_1}\] 
    is injective. Then from $f,\psi$ we can define a map fitting into a commutative diagram 
    \begin{center}
        \begin{tikzcd}
            f^*\mathbb{E}_2 \arrow{d}{f^*\phi_2} \arrow[dashed]{r} & \mathbb{E}_1 \arrow{d}{\phi_1}\\
            f^*\mathbb{L}_{Z_2} \arrow{r} & \mathbb{L}_{Z_1}
        \end{tikzcd}
    \end{center}
    such that the induced map of cones 
    \[\mathbb{E}_{1,2}:=\operatorname{cone}(f^*\mathbb{E}_2\to \mathbb{E}_1)\to \operatorname{cone}(f^*\mathbb{L}_{Z_2}\to \mathbb{L}_{Z_1})=\mathbb{L}_{Z_1/Z_2}\]
    is a relative perfect obstruction theory.
\end{lemma}
\begin{proof}
The condition $\psi s_1=s_2f$ implies in particular that $f$ restricts to a map of schemes 
\[f:Z_1\to Z_2.\]
The condition further implies that 
\begin{center}
    \begin{tikzcd}
        f^*(E_2^{\vee}|_{Z_2}) \arrow{d}{\text{d}f^*s_2^{\vee}} \arrow{r}{\psi^\vee} & E_1^{\vee}|_{Z_1} \arrow{d}{\text{d}s_1^{\vee}} \\
        f^*(\Omega_{Y_2}|_{Z_2}) \arrow{r}{df} & \Omega_{Y_1}|_{Z_1}
    \end{tikzcd}
\end{center}
commutes giving us the desired map $f^*\mathbb{E}_2\to \mathbb{E}_1$ fitting into the commutative diagram. Since $f^*\mathbb{E}_2,\mathbb{E}_1$ are perfect of amplitude $[-1,0]$ we have that the cone $\mathbb{E}_{1,2}$ is perfect of amplitude $[-2,0]$. It is explicitly given by the complex
\[f^*(E_2^{\vee}|_{Z_2})\xrightarrow{\left(\begin{array}{c}
        -\text{d}s_2^{\vee}   \\
         \psi^{\vee}
    \end{array}\right)} f^*(\Omega_{Y_2}|_{Z_2})\oplus E_1^{\vee}|_{Z_1}\xrightarrow{\left(\text{d}f,\text{d}s_1^{\vee}\right)}\Omega_{Y_1}|_{X_1}\] 
so that the transversality condition exactly implies that $h^{-2}(\mathbb{E}_{1,2})=0$. Finally, taking the map of long exact sequences induced by the map of triangles we get the commuting diagram
\begin{center}
    \begin{tikzcd}[column sep=small]
        h^{-1}(\mathbb{E}_{1}) \arrow{r} \arrow[two heads]{d} & h^{-1}(\mathbb{E}_{1,2}) \arrow{r} \arrow{d} &h^{0}(f^*\mathbb{E}_{2})\arrow{r} \arrow{d}{\rotatebox{90}{$\sim$}} &h^{0}(\mathbb{E}_{1}) \arrow{r} \arrow{d}{\rotatebox{90}{$\sim$}} & h^{0}(\mathbb{E}_{1,2}) \arrow{r} \arrow{d} & 0 \arrow[equal]{d} \arrow{r} & 0 \arrow[equal]{d}\\
        h^{-1}(\mathbb{L}_{Z_1}) \arrow{r} & h^{-1}(\mathbb{L}_{Z_1/Z_2}) \arrow{r} &h^{0}(f^*\mathbb{L}_{Z_2})\arrow{r} &h^{0}(\mathbb{L}_{Z_1}) \arrow{r} & h^{0}(\mathbb{L}_{Z_1/Z_2}) \arrow{r} & 0 \arrow{r} &0 \\
    \end{tikzcd}
\end{center}
and from the 4-lemma we conclude that $h^{-1}(\mathbb{E}_{1,2})\to h^{-1}(\mathbb{L}_{Z_1/Z_2})$ is an epimorphism and from the 5-lemma we conclude that $h^0(\mathbb{E}_{1,2})\to h^0(\mathbb{L}_{Z_1/Z_2})$ is an isomorphism. We conclude that $\mathbb{E}_{1,2}\to \mathbb{L}_{Z_1/Z_2}$ is a perfect obstruction theory as desired.
\end{proof}
\begin{remark}
    Note that the transversality requirement is automatic if $\psi$ is surjective. In this case we get a short exact sequence 
    \[0\to K\to E_1\to f^*E_2\to 0\]
    for $K=\operatorname{ker}\psi$. This yields the identity
    \[\wedge_{-1}^{\bullet}[E_1^{\vee}]=(\wedge_{-1}^{\bullet}[K^{\vee}])\cdot (f^*\wedge_{-1}^{\bullet}[E_2^{\vee}])\]
    in $K_0(Y)$ and so in particular by the deifnition of the virtual structure sheaf we get that
    \[\oo_{Z_1}^{\operatorname{vir}}=(\wedge_{-1}^{\bullet}[K^{\vee}|_{Z_1}])\cdot \oo_{Z_2}^{\operatorname{vir}}\]
    in $K_0(Z_1)$.
\end{remark}
\subsection{Fixed points of the nested Hilbert scheme of points}
Recall that for the induced $T=\Gm^n$ action on $\operatorname{NHilb}^{\underline{d}}(\A^n)$ the fixed locus is finite and reduced and corresponds to nestings of monomial ideals. These ideals are in particular in bijection with higher dimensional Young diagrams. We briefly recall the definition.
\begin{definition}
    An $n$-dimensional Young diagram of size $d$ is a subset $\lambda\subset \Z^n_{\geq 0}$ which satisfies that for all $\mathbf{a}=(a_1,\dots, a_n)\in \Z^{n}_{\geq 0}$ and $\mathbf{b}=(b_1,\dots, b_n)\in \lambda$ then $a_i\leq b_i$ for $i=1,\dots, n$ implies $a\in \lambda$. A nested $n$-dimensional Young diagram $\lambda_\bullet$ of size $\underline{d}=(d_0,\dots, d_r)$ is a chain
    \[\varnothing=\lambda_0\subset \lambda_1\subset \cdots \subset \lambda_{r+1}\]
    such that $\lambda_i$ is an $n$-dimensional Young diagram for $i=1,\dots, r+1$ and such that 
    \[|\lambda_{i+1}|-|\lambda_i|=d_i.\]
\end{definition}
We recall that for a nested $n$-dimensional Young diagram of size $\underline{d}$ we get a nesting of monomial ideals
\[I_{r+1}\subset \cdots \subset I_1\]
by setting
\[I_i=I_{\lambda_i}=\langle \mathbf{x}^{\mathbf{a}}\;|\; \mathbf{a}\notin \lambda_i\rangle\subset \C[x_1,\dots, x_n].\]
This bijectively describes all $T$-fixed points of $\operatorname{NHilb}^{\underline{d}}(\A^n)$ when $\lambda_{\bullet}$ ranges over all $n$-dimensional nested Young diagrams of size $\underline{d}$. We recall that with our convention of torus action we have 
\[[\mathcal{O}_{Z_i}]=[\C[x_1,\dots, x_n]/I_{\lambda_i}]=\sum_{a\in \lambda_i}\mathbf{t}^{\mathbf{a}}\in K_0^T(\operatorname{pt})\cong \Z[t_1^{\pm},\dots, t_n^{\pm}].\]
\section{The non-commutative nested Hilbert scheme}
In this section, we will show the existence of smooth moduli spaces which we call the nested non-commutative Hilbert scheme. The non-commutative Hilbert scheme was first introduced by Nori \cite{Nori1977Appendix} and has been used successfully to derive various results on the Hilbert scheme of points. A nested version for the non-commutative Hilbert scheme with a full dimension vector (i.e. where the difference in lengths of the subschemes of the nesting are constant 1) on $\C^2$ was considered in \cite{Oblomkov2018-ww} using a different construction to avoid non-reductive quotients. 

This section is devoted to defining the nested non-commutative Hilbert scheme on $\C^n$ for any dimension vector. We will introduce it as the quotient stack of a parabolic subgroup acting freely on a smooth variety. This immediately implies that the corresponding stack is smooth and an algebraic space. We will then identify its functor of points and provide explicit Zariski coverings, taking our inspiration from the proof of the representability of the Hilbert scheme of points found in \cite{GUSTAVSEN2007705}. We can then construct an explicit embedding from the nested Hilbert scheme of points using the functor of points.

\subsection{Constructions}
For this section, we fix $r\geq 0$, a dimension vector $\underline{d}=(d_0,\dots, d_{r})\in \Z_{\geq 0}^{r+1}$ and a reference flag of $\C$-vectorspaces 
\[N_\bullet= (N_{0}\supset N_{1}\supset \cdots \supset N_{r}=0)\]
such that $\operatorname{dim}_{\C}N_i/N_{i+1}=d_i$. We will denote by $\Hom^{\operatorname{fil}}_{\C}$ and $\operatorname{End}^{\operatorname{fil}}_{\C}$ the flag preserving $\C$-linear homomorphisms and endomorphisms. We let $P_{\underline{d}}\subset \operatorname{GL}(N_0)$ be the parabolic subgroup of flag-preserving automorphisms. 
\begin{definition}
    We define 
    \[\widetilde{\operatorname{ncNHilb}^{\underline{d}}}(\A^n)\subset \operatorname{End}_{\C}^{\operatorname{fil}}(N_\bullet)^{\oplus n}\oplus N_0\]
    given by the Zariski open locus of tuples $(A_1,\dots, A_n,v)$ where 
    \[\C\langle A_1,\dots, A_n\rangle v=N_0. \]
    Equivalently, it consists of the tuples $(A_1,\dots, A_n,v)$ where the smallest subspace of $N_0$ closed under $A_1,\dots, A_n$ and containing $v$ is $N_0$ itself.
\end{definition}

\begin{definition}
    We let $P_{\underline{d}}$ act on $\operatorname{End}_{\C}^{\operatorname{fil}}(N_\bullet)^{\oplus n}\oplus N_0$ by 
    \begin{align*}
        g.(A_1,\dots, A_n,v)=(A_1^g,\dots , A_n^g,gv)
    \end{align*}
    where $A_i^g=gA_ig^{-1}$.
\end{definition}
\begin{lemma}
    The action of $P_{\underline{d}}$ restricts to a free action on $\widetilde{\operatorname{ncNHilb}^{\underline{d}}}(\A^n)$.
\end{lemma}
\begin{proof}
    Let $(A_1,\dots, A_n,v)\in \widetilde{\operatorname{ncNHilb}^{\underline{d}}}(\A^n)$ and let $g\in P_{\underline{d}}$. If $M\subset N_0$ is invariant under $A_1^g,\dots, A_n^g$ and contains $gv$ then $g^{-1}M$ is invariant under $A_1,\dots, A_n$ and $v\in g^{-1}M$ hence by stability $g^{-1}M=N_0$ and so $M=N_0$ proving that $(A_1^g,\dots, A_n^g,gv)$ is stable. This shows that $\widetilde{\operatorname{ncNHilb}^{\underline{d}}}(\A^n)$ is closed under the $P_{\underline{d}}$ action. If $(A_1^g,\dots, A_n^g,gv)=(A_1,\dots, A_n,v)$ then $\operatorname{ker}(g-\operatorname{id}_{N_0})$ contains $v$ and is invariant under $A_1,\dots A_n$ hence $g=\operatorname{id}$ thus proving that the action is free. 
\end{proof}

\begin{definition}
    We define 
    \[ \operatorname{ncNHilb}^{\underline{d}}(\A^n)=\left[ \widetilde{\operatorname{ncNHilb}^{\underline{d}}}(\A^n)\big{/}P_{\underline{d}}\right]\]
    which we call the \textit{the non-commutative nested Hilbert scheme}. Here $[X/G]$ denotes the resulting quotient stack. Note that since we are using a smooth algebraic group acting freely on a smooth space the resulting quotient is a smooth algebraic space. When $r=1$ so that $\underline{d}=d_0$ we will drop the "N" prefix in the notation and simply write $\operatorname{ncHilb}^{d_0}(\A^n)$ which we will simply refer to as \textit{the non-commutative Hilbert scheme}. In this case it agrees with the original construction, see for example \cite{Reineke2005-xv}
\end{definition}
Next we identify the functor of points by making the obvious generalization of the stable data above from vector spaces to vector bundles. 
\begin{definition}
    Let $S\in \operatorname{Sch}_{/\C}$. A \textit{non-commutative stable tuple} of type $\underline{d},n$ on $S$ consists of the data $(\mathcal{F}_\bullet, \phi_1,\dots, \phi_n,s)$ where 
    \begin{enumerate}
        \item $\mathcal{F}_\bullet$ is a flag of locally free sheaves
        \[\mathcal{F}_\bullet=(\mathcal{F}_0\supset \cdots \supset \mathcal{F}_m=0)\]
        such that $\mathcal{F}_i/\mathcal{F}_{i+1}$ is locally free of rank $d_i$.
        \item $\phi_1,\dots, \phi_n$ are operators 
        \[\phi_1,\dots, \phi_n: \mathcal{F}_0\to \mathcal{F}_0\]
        such that $\phi_i(\mathcal{F}_j)\subset \mathcal{F}_j$.
        \item $s\in H^0(S,\mathcal{F}_0)$ is a global section such that the induced map 
        \begin{center}
            \begin{tikzcd}
                \oo_S\langle x_1,\dots, x_n\rangle\arrow{r} & \mathcal{F}_0 \\[-20pt]
                f(x_1,\dots, x_n) \arrow[mapsto]{r} & f(\phi_1,\dots, \phi_n)s
            \end{tikzcd}
        \end{center}
    is an epimorphism. 
    \end{enumerate}
    A \textit{framed non-commutative stable tuple} of type $(\underline{d},n)$ is the data of $(\mathcal{F}_\bullet, \phi_1,\dots, \phi_n,s,\beta)$ where $(\mathcal{F}_\bullet, \phi_1,\dots, \phi_n,s)$ is a non-commutative stable tuple and 
    \[\beta: \oo_S\otimes_{\C} N_\bullet \to \mathcal{F}_\bullet\]
    is an isomorphism of flags. Two non-commutative stable tuples $(\mathcal{F}_\bullet, \phi_1,\dots, \phi_n,s)$ and $(\mathcal{F}_\bullet', \phi_1',\dots, \phi_n',s')$ (respectively framed non-commutative stable tuples $(\mathcal{F}_\bullet, \phi_1,\dots, \phi_n,s,\beta)$ and $(\mathcal{F}_\bullet', \phi_1',\dots, \phi_n',s',\beta')$) are said to be equivalent if there exists an isomorphism of flags $\rho: \mathcal{F}_\bullet \to \mathcal{F}_\bullet'$ such that $\phi_i'=\rho \phi_i\rho^{-1}$ and $s'=\rho s$ (respectively also $\beta'=\rho \beta$).
    \end{definition}

\begin{definition}
    We define the functor
    \[nc\mathcal{H}ilb^{\underline{d},n},: \operatorname{Sch}_{/\C}^{\operatorname{op}}\to \operatorname{Set}\]
    as follows. For $S\in \operatorname{Sch}_{/\C}$ we let 
    \begin{align*}
        nc\mathcal{H}ilb^{\underline{d},n}(S)=\{(\mathcal{F}_\bullet, \phi_1,\dots, \phi_n,s) \text{ non-commutative stable tuple} \}/\sim
    \end{align*}
    and for $f:S'\to S$ we simply define 
    \[f^*(\mathcal{F}_\bullet, \phi_1,\dots, \phi_n,s)=(f^*\mathcal{F}_\bullet, f^*\phi_1,\dots, f^*\phi_n,f^*s).\]
    We similarly define a framed version
       \[\widetilde{nc\mathcal{H}ilb^{\underline{d},n}}: \operatorname{Sch}_{/\C}^{\operatorname{op}}\to \operatorname{Set}\]
    by
    \[\widetilde{nc\mathcal{H}ilb^{\underline{d},n}}(S)=\{(\mathcal{F}_\bullet, \phi_1,\dots, \phi_n,s,\beta) \text{ framed non-commutative stable tuple} \}/\sim\]
    and 
    \[f^*(\mathcal{F}_\bullet, \phi_1,\dots, \phi_n,s,\beta)=(f^*\mathcal{F}_\bullet, f^*\phi_1,\dots, f^*\phi_n,f^*s,f^*\beta).\]
\end{definition}
\begin{prop}
    The map
        \begin{align*}
        \widetilde{nc\mathcal{H}ilb^{\underline{d},n}}\to nc\mathcal{H}ilb^{\underline{d},n}
    \end{align*}
    given by forgetting the framing is isomorphic to the functor of points of the quotient map 
    \begin{align*}
        \widetilde{\operatorname{ncNHilb}^{\underline{d}}}(\A^n)\to \operatorname{ncNHilb}^{\underline{d}}(\A^n).
    \end{align*}
\end{prop}
\begin{proof}
    Set $X= \operatorname{End}_{\C}^{\operatorname{fil}}(N_\bullet)^{\oplus n} \oplus N_0$ viewing the vector space as a scheme (i.e. by taking $\spec \sym^\bullet (-)^\vee$). Let $V_\bullet=N_\bullet \otimes_{\C} \oo_X$ which is a flag of locally free sheaves on $X$. We construct universal maps on $V_\bullet$ as follows. We set $d=|\underline{d}|=d_0+\cdots +d_r$ and choose a basis $v_1,\dots, v_d$ of $N_0$ such that 
    \[\operatorname{span}(v_{d_0+\cdots +d_{i-1}+1},\dots,v_d)=N_i.\]
    For $1\leq i \leq d$ define we $w(i)=l$ if 
    \[d_0+\cdots d_{l-1}<i\leq d_0+\cdots +d_l\]
    i.e. if and only if $v_i\in N_l\backslash N_{l+1}$. Writing everything in the coordinates of the basis we get
    \[X\cong \spec \C[a_{i,j}^k,\alpha_m]\]
    where the indices ranges over $1\leq k\leq n$, $1\leq i,j\leq d$ with $w(j)\leq w(i)$ and $1\leq m\leq d$. Explicitly, for a point $(A_1,\dots, A_n,v)\in X$ we have 
    \begin{align*}
   A_k(v_j)=\sum_{i}a_{i,j}^kv_i, && v=\sum_m\alpha_mv_m.
    \end{align*}
    We define
    \begin{align*}
      \widetilde{\Phi}_k:V_0\to V_0, && \widetilde{\mathfrak{v}}:\oo_X\to V_0
    \end{align*}
    by 
    \begin{align*}
        \widetilde{\Phi}_k(v_j\otimes 1)=\sum_{i} v_i\otimes a_{i,j}^k,  && \widetilde{\mathfrak{v}}(1)=\sum_{m}v_m\otimes \alpha_m.
    \end{align*} 
    Note also that the identity map $\beta_0:=\operatorname{id}_{V_0}$ is a framing of $V_\bullet$. By definition $(V_\bullet, \widetilde{\Phi}_1,\dots, \widetilde{\Phi}_n,\widetilde{\mathfrak{v}},\beta_0)$ is a framed non-commutative tuple when restricted to $\widetilde{\operatorname{ncNHilb}^{\underline{d}}}(\A^2)$, which then in turn classifies a map 
    \[\widetilde{\operatorname{ncNHilb}^{\underline{d}}}(\A^2)\to \widetilde{nc\mathcal{H}ilb^{\underline{d},n}}.\]
    Conversely, suppose we have a map $S\to \widetilde{nc\mathcal{H}ilb^{\underline{d},n}}$ classified by a framed non-commutative tuple $(\mathcal{F}_\bullet,\varphi_1,\dots, \varphi_n,s,\beta)$. Then we can view $(\beta^{-1}\varphi_1\beta,\dots ,\beta^{-1}\varphi_n\beta ,\beta^{-1}s)$ as a map 
    \[\oo_S\to \left(\operatorname{End}^{\operatorname{fil}}(N_\bullet)^{\oplus n} \oplus N_0\right)\otimes_{\C} \oo_S\]
    which in turn classifies a map $f:S\to X$ in such a way that 
    \[f^*(V_\bullet, \widetilde{\Phi}_1,\dots, \widetilde{\Phi}_n,\widetilde{\mathfrak{v}},\beta_0)\sim (\mathcal{F}_\bullet,\varphi_1,\dots, \varphi_n,s ,\beta)\]
    via $\beta:f^*V_\bullet=N_\bullet\otimes_{\C}\oo_S\to \mathcal{F}_\bullet$. The various conditions on being a non-commutative tuple exactly translates to $f$ factoring through $\widetilde{\operatorname{ncNHilb}^{\underline{d}}}(\A^n)$, giving us the desired inverse. \\

    \noindent Next we note that $V_\bullet= N_\bullet \otimes_{\C} \oo$ on $\widetilde{\operatorname{ncNHilb}^{\underline{d}}}(\A^2)$ has a natural $P_{\underline{d}}$-equivariant structure by letting $P_{\underline{d}}=\operatorname{Aut}_{\C}^{\operatorname{fil}}(N_\bullet)$ act by multiplication on the fibers. This structure naturally makes $\widetilde{\Phi}_1,\dots, \widetilde{\Phi_n}$ and $\widetilde{\mathfrak{v}}$ equivariant. By properties of the quotient stack, pullback along the quotient induces an equivalence of categories 
    \[\operatorname{QCoh}(\operatorname{ncNHilb}^{\underline{d}}(\A^2))\cong\operatorname{QCoh}^{\operatorname{P_{\underline{d}}}}(\widetilde{\operatorname{ncNHilb}^{\underline{d}}}(\A^2))\]
    where the left hand side is equivariant quasi-coherent sheaves. From this we get an induced non-commutative tuple $(\mathcal{V}_\bullet, \Phi_1,\dots, \Phi_n,\mathfrak{v})$ on $\operatorname{ncNHilb}^{\underline{d}}(\A^2)$ such that $\pi^*(\mathcal{V}_\bullet, \Phi_1,\dots, \Phi_n,\mathfrak{v})=(V_\bullet, \widetilde{\Phi}_1,\dots, \widetilde{\Phi}_n,\widetilde{\mathfrak{v}})$, which in turn gives us a map 
    \[\operatorname{ncNHilb}^{\underline{d}}(\A^n)\to nc\mathcal{H}ilb^{\underline{d},n}.\]
    Conversely, given a map $S\to nc\mathcal{H}ilb^{\underline{d},n}$ classified by a non-associative tuple $(\mathcal{F}_\bullet, \varphi_1,\dots,\varphi_2,s)$ we can pick an open cover $\{U_i\}_{i}$ of $S$ and trivializations $\beta_i: \oo_{U_i}\otimes_{\C}N_\bullet\to \mathcal{F}_{\bullet}|_{U_i}$. It follows that $(\mathcal{F}_{\bullet}|_{U_i},\varphi_1,\dots, \varphi_2,s,\beta_i)$ is a framed non-associative tuple. Using the above, these classify maps
    \[f_i:U_i\to \widetilde{\operatorname{ncNHilb}^{\underline{d}}}(\A^2).\]
    On overlaps $U_{i,j}=U_i\cap U_j$ the transition function of $\beta_i,\beta_j$ determines a map
    \[\gamma_{j,i} :U_{j,i}\to P_{\underline{d}}\]
    such that $\gamma_{j,i}.\beta_i=\beta_j$. It follows that
    \[(\pi f_i)|_{U_{i,j}}=(\pi f_j)|_{U_{i,j}}\]
    so that $\{\pi f_i\}_i$ glues together to a map 
    \[S\to \operatorname{ncNHilb}^{\underline{d}}(\A^2)\]
    giving us the desired inverse.
\end{proof}
\begin{remark}
    We note that the proposition shows the existence of a universal non-commutative tuple on $\operatorname{ncNHilb}^{\underline{d}}(\A^n)$. We will denote the universal tuple by $(\mathcal{V}_\bullet, \Phi_1,\dots, \Phi_n,\mathfrak{v})$
\end{remark}
It now only remains to show that our moduli space is representable by a scheme. Since it is an algebraic space it satisfies fppf descent and hence also Zariski descent. It therefore only remains to construct a Zariski open cover of schemes. We note that the universal tuple in turn induces a universal epimorphism
\[ q_{nc}:\C\langle x_1,\dots, x_n \rangle\otimes_{\C}\oo\cong \oo \langle x_1,\dots, x_n \rangle\twoheadrightarrow \mathcal{V}_0\]
on $\operatorname{ncNHilb}^{\underline{d}}(\A^n)$. Inspired by \cite{GUSTAVSEN2007705} we consider loci of sections of these epimorphisms.

\begin{definition}
    For fixed $\gamma:N_0\to \C\langle x_1,\dots, x_n \rangle$ we let
    \begin{align*}
        \operatorname{ncNHilb}_{\gamma}^{\underline{d}}(\A^n)\subset\operatorname{ncNHilb}^{\underline{d}}(\A^n)
    \end{align*}
    be the loci where the composition 
    \[N_0\otimes_{\C} \oo \xrightarrow{\gamma\otimes \id} \oo\langle x_1,\dots, x_n\rangle \overset{q_{nc}}{\twoheadrightarrow} \mathcal{V}_0 \]
    is an isomorphism (or equivalently surjective). Note that this is a Zariski open as it is defined as the surjectivity locus of a map of sheaves. 
\end{definition}
\begin{prop}
    Let $\gamma:N_0\to \C\langle x_1,\dots, x_n \rangle$. Then $\operatorname{ncNHilb}_{\gamma}^{\underline{d}}(\A^n)$ is representable by a scheme.
\end{prop}
\begin{proof}
    Consider the flag variety $\operatorname{Flag}_{\underline{d}}(N_0)$ which is the moduli of flags on $N_0$ of dimension prescribed by $\underline{d}$. Let $d=d_0+\cdots +d_r$ and consider $\widetilde{\operatorname{ncHilb}^d}(\A^n)$ i.e. the non-nested version. We have a map
    \[\operatorname{ncNHilb}_{\gamma}^{\underline{d}}(\A^n)\to \operatorname{Flag}_{\underline{d}}(N_0)\times \widetilde{\operatorname{ncHilb}^d}(\A^n)\]
    which on the functor of points for $S\in \operatorname{Sch}_{/\C}$ is given by
    \[(\mathcal{F}_\bullet, \phi_1,\dots, \phi_n,s) \mapsto \left((q_{nc}(\gamma\otimes 1))^{-1}\mathcal{F}_\bullet,(\mathcal{F}_0,\phi_1,\dots, \phi_n,s,q_{nc}(\gamma\otimes \id))\right)\]
     i.e. in the second coordinate we forget the flag and use $q_{nc}(\gamma\otimes 1)$ as framing while we in the first coordinate use the flag structure induced on $N_0\otimes_{\C} \mathcal{\oo}_S$ by $\mathcal{F}_\bullet$ via $q_{nc}(\gamma\otimes 1)$. The target of this map is a scheme, so we are done if we can show that the map is a closed immersion. It is clearly a monomorphism on the functor of points. We can identify the image of the map as the locus of 
     \[\left(\mathcal{G}_\bullet,(\mathcal{F}_0,\phi_1,\dots,\phi_n,s,\beta)\right)\]
     where $\phi_1,\dots, \phi_n$ are flag-preserving wrt. the flag structure $\beta(\mathcal{G}_\bullet)$ on $\mathcal{F}_0$ and where $\beta=q_{nc}(\gamma_1\otimes \id)$. These conditions are Zariski closed, thus finishing the proof. 
\end{proof}
\begin{prop}
    The subschemes $\{\operatorname{ncNHilb}_{\gamma}^{\underline{d}}(\A^n)\}_{\gamma \in \operatorname{Hom}_{\C}(N_0,\C\langle x_1,\dots, x_n\rangle)}$ provide a Zariski cover of $\operatorname{ncNHilb}^{\underline{d}}(\A^n)$. In particular, $\operatorname{ncNHilb}^{\underline{d}}(\A^n)$ is a scheme.
\end{prop}
\begin{proof}
    This follows simply because the maps $q_{nc}$ admit sections $\gamma$ Zariski locally. Indeed, for each closed point $p\in \operatorname{ncNHilb}^{\underline{d}}(\A^n)$ the map is given by
    \[q_{nc}|_{p}: \C\langle x_1,\dots, x_n \rangle \twoheadrightarrow \mathcal{V}_{0,p_1}\]
    for which such a section clearly exists.
\end{proof}
\subsection{Embedding of the nested Hilbert scheme}
Consider the nested Hilbert scheme of points $\operatorname{NHilb}^{\underline{d}}(\A^n)$. The universal subschemes 
\[\mathcal{Z}_1\subset \mathcal{Z}_2\subset \cdots \subset \mathcal{Z}_{r+1}\subset \A^n_{\C}\times \operatorname{NHilb}^{\underline{d}}(\A^n)\cong \A^n_{\operatorname{NHilb}^{\underline{d}}(\A^n)}\]
give rise to surjections of locally free $\oo=\oo_{\operatorname{NHilb}^{\underline{d}}(\A^n)}$-algebras
\[\oo[x_1,\dots, x_n]\twoheadrightarrow \operatorname{pr}_*\oo_{\mathcal{Z}_r}\twoheadrightarrow\cdots \twoheadrightarrow \operatorname{pr}_*\oo_{\mathcal{Z}_{1}}.\]
We consider the flag $\mathcal{F}_\bullet$ given by setting $\mathcal{F}_0=\operatorname{pr}_*\oo_{\mathcal{Z}_{r+1}}$ and
\[\mathcal{F}_i=\operatorname{ker}(\operatorname{pr}_*\oo_{\mathcal{Z}_r}\twoheadrightarrow\operatorname{pr}_*\oo_{\mathcal{Z}_i}).\]
We consider the map
\[\oo\langle x_1,\dots, x_ n\rangle\twoheadrightarrow \oo[x_1,\dots, x_n]\twoheadrightarrow \mathcal{F}_{\bullet}\]
which induces a non-commutative tuple which in turn classifies a map to the non-commutative nested Hilbert scheme.
\begin{prop}
    The map
    \begin{align*}
    \operatorname{NHilb}^{\underline{d}}(\A^n)\to \operatorname{ncNHilb}^{\underline{d}}(\A^n).
\end{align*}
classified by the tuple above is a closed immersion. Its image agrees with the locus where $\Phi_1,\dots, \Phi_n$ commute which we will call the \textit{the commutativity locus}.
\end{prop}
\begin{proof}
    We construct an inverse from the commutativity locus via the functor of points. The commutativity locus simply classifies tuples where the associated quotient map $q_{nc}$ factors through $\oo_S[x_1,\dots, x_n]$. In this case an $S$-valued point is therefore completely classified up to isomorphism by a flag $\mathcal{F}_\bullet$ of locally free sheaves with an epimorphism
    \[q:\oo_S[x_1,\dots, x_n]\twoheadrightarrow\mathcal{F}_0\]
    such that $q^{-1}(\mathcal{F}_i)$ is closed under left muliplication by $x_1,\dots, x_n$, i.e. is a left and therefore also a two-sided ideal. This in turn induces a series of epimorphisms
    \[\oo_S[x_1,\dots, x_n]\twoheadrightarrow\mathcal{F}_0\twoheadrightarrow\mathcal{F}_0/\mathcal{F}_{r-1}\twoheadrightarrow\cdots \twoheadrightarrow\mathcal{F}_0/\mathcal{F}_1\]
    which allows us to identify $\mathcal{F}_{0}/\mathcal{F}_i$ with a locally free $\oo_S$-algebra $\oo_{Z_i}$ giving us the desired $S$-valued point
    \[Z_1\subset \cdots Z_r\subset \A^n_{\C}\times S\]
    on $\operatorname{NHilb}^{\underline{d}}(\A^n)$. One can easily check that these constructions are the inverse of each other. Clearly the commutativity locus is Zariski closed as it can be given by the zero locus of the commutators
    \[\bigoplus_{1\leq i<j\leq n} [\Phi_i,\Phi_j]\in H^0\left(\textstyle{\bigoplus_{i,j}}\mathcal{E}nd^{\operatorname{fil}}(\mathcal{V}_\bullet)\right).\]
\end{proof}
\begin{remark}
Recall that for a quasi-coherent sheaf $\mathcal{F}\in \operatorname{QCoh}(\A^n)$ we can construct tautological sheaves 
\[\mathcal{F}^{[d_0+\cdots+d_r]}\twoheadrightarrow\mathcal{F}^{[d_0+\cdots +d_{r-1}]}\twoheadrightarrow \cdots \twoheadrightarrow\mathcal{F}^{[d_0]}\]
on $\operatorname{NHilb}^{\underline{d}}(\A^n)$ by setting
\[\mathcal{F}^{[d_0+\cdots+d_{i-1}]}=p_{i*}q_i^*(\mathcal{F})\]
where 
\[\A^n_{\C^n}\xleftarrow{q_i}\mathcal{Z}_i\xrightarrow{p_i}\operatorname{NHilb}^{\underline{d}}(\A^n)\]
are the projection maps of the universal family. We note that by construction we have 

\[\mathcal{V}_0|_{\operatorname{NHilb}^{\underline{d}}}= \mathcal{O}_{\A^n}^{[d_0+\cdots d_r]}\]
and for each $i\geq 1$ we have a short exact sequence
\[0\to \mathcal{V}_i|_{\operatorname{NHilb}^{\underline{d}}(\A^n)}\to \mathcal{O}_{\A^n}^{[d_0+\cdots d_r]}\to \mathcal{O}_{\A^n}^{[d_0+\cdots d_{i-1}]}\to 0\]
where $\mathcal{V}_\bullet$ is the universal flag on $\operatorname{ncNHilb}^{\underline{d}}(\A^n)$.
\end{remark}
\subsection{The torus action}
We let $T=\Gm^n$ which we let act on $\A^n$ by
\[(t_1,\dots, t_i).(p_1,\dots, p_i)=(t_1^{-1}p_1,\dots, t_n^{-1}p_n). \]
This equips $\C[x_1,\dots, x_n]$ with a natural $T$ space structure where $t.x_i=t_i^{}x_i$ and $t.1=1$. With this in mind, we let $T$ act on $\widetilde{\operatorname{ncNHilb}^{\underline{d}}}(\A^n)$ by 
\[t.(A_1,\dots, A_n,v)=(t_1^{-1}A_1,\dots, t_n^{-1}A_n,v).\]
This commutes with the $P_{\underline{d}}$-action and therefore descends to an action on the quotient. Similarly, equipping the $P_{\underline{d}}$-equivariant sheaf $V_\bullet$ with the trivial $T$-action gives us an action on the corresponding universal flag $\mathcal{V}_\bullet$. In this way, the universal operators are equivariant maps 
\[\Phi_i: t_i\mathcal{V}_\bullet \to \mathcal{V}_\bullet\]
where the twist $t_i$ is by tensoring with the one-dimensional representation of character $t_i$. We can therefore view it as a section of 
\[t_i^{-1}\mathcal{E}nd^{\operatorname{fil}}(\mathcal
V_\bullet)\]
The universal cyclic vector gives an equivariant map
\[\mathfrak{v}:\oo\to \mathcal{V}_{\bullet}\]
where $\oo$ is given the trivial $T$-equivariant structure. 
We equip $\operatorname{NHilb}^{\underline{d}}(\A^n)$ with the $T$-action induced from the action on $\A^n$. With these actions the embedding
\[\operatorname{NHilb}^{\underline{d}}(\A^n) \hookrightarrow \operatorname{ncNHilb}^{\underline{d}}(\A^n)\]
and the maps between equivariant sheaves 
\[ \mathcal{V}_i|_{\operatorname{NHilb}^{\underline{d}}(\A^n)}\hookrightarrow \mathcal{O}_{\A^n}^{[d_0+\cdots d_r]}\]
are also equivariant.
\section{The commutativity bundle and the perfect obstruction theory}
\subsection{Reduction of commutativity equations}
Naively the set of commutativity equations can be structured into a section of the equivariant bundle 
\[\bigoplus_{1\leq i<j\leq n}t_{i}^{-1}t_j^{-1}\mathcal{E}nd^{\operatorname{fil}}(\mathcal{V}_\bullet).\]
Letting 
\[h(\underline{d})=\sum_{0\leq i\leq j \leq r} d_id_j=\operatorname{rk}\mathcal{E}nd^{\operatorname{fil}}(\mathcal{V}_\bullet)\]
and
\[d=\sum_{i=0}^rd_i\]
we see that
\[\operatorname{dim}\operatorname{ncNHilb}^{\underline{d}}(\A^n)=n\cdot h(\underline{d})+d-h(\underline{d}) \]
giving us a virtual dimension of 
\[d-\frac{(n-1)(n-2)}{2}h(\underline{d}).\]
For $n=2$ this gives us a virtual dimension of $d$. For $n\geq 3$ the virtual dimension becomes negative since 
\[\frac{d(d+1)}{d}\leq h(\underline{d})\leq d^2.\]
In this section we show that the rank of this bundle can be reduced in the $n=2$ case giving us a virtual dimension of $d+d_0$. We note that in the non-nested case $r=0$ the Hilbert scheme $\operatorname{Hilb}^{d}(\A^2)$ is smooth of dimension $2d$. In particular, the above virtual class of virtual dimension $d$ is not the fundamental one. This means that we are able to locally reduce the number of commutativity equations by $d$. We will show that the cyclicity condition allows us to do this globally. We do so by showing that if the commutator is in the span of the cyclic vector then it is forced to vanish. 

\begin{lemma}
    Let $V$ be a vector space with operators $A_1,A_2\in \operatorname{End}(V)$ and $v\in V$ cyclic. If $\im ([A_1,A_2])\subset \operatorname{span}\{v\}$ then $[A_1,A_2]=0$.
\end{lemma}
\begin{proof}
Set $B=[A_1,A_2]$. We can find a functional $\phi:V\to \C$ such that $Bw=\phi(w)v$. Let 
\[V_m=\operatorname{span}\{A_{i_1}\cdots A_{i_k}v \: | \: k\leq m,\: i_j\in \{1,2\}\}.\]
We have $V=\bigcup_{m\geq 0}V_m$ and show by induction that $\phi|_{V_m}\equiv 0$. We note that for any operator $L:V\to V$ it holds that 
\[\operatorname{tr}(BL)=\phi(Lv)\]
since we have $BL(w)=\phi(Lw)v$. In particular, 
\[\phi(v)=\operatorname{Tr}[A_1,A_2]=0\]
proving that $\phi|_{V_0}\equiv0$. Suppose now that $\phi|_{V_m}\equiv 0$ for $m\geq 0$. Let $L=A_{i_1}\cdots A_{i_{m+1}}$ so that $Lv\in V_{m+1}$. We first prove that $Lv=A_2^kA_1^lv$ for some $l,k$ with $l+k=m+1$. Indeed, if $L$ contains an inversion we can write it in the form $L=FA_1A_2G$ where $F,G$ are non-commutative words in $A_1,A_2$ of length $\leq m$. Then 
\[FA_1A_2Gv=F(A_2A_1+B)Gv=FA_2A_1Gv+FBGv=FA_2A_1Gv-\phi(Gv)Fv\]
but $\phi(Gv)=0$ by the induction hypothesis, so we have decreased the number of inversions by 1. Inductively continuing this procedure allows us to write $Lv=A_2^kA_1^l$. If $k=0$ then
\[\phi(Lv)=\operatorname{tr}([A_1,A_2]A_1^{l})=\operatorname{tr}(A_1A_2A_1^l)-\operatorname{tr}(A_2A_1A_1^l)\]
which disappears by invariance of the trace under cyclic shifts. For $k>0$ we first see that
\[\operatorname{tr}([L,A_1]A_2)=\operatorname{tr}(LA_1A_2)-\operatorname{tr}(A_1LA_2)=\operatorname{tr}(BL)=\phi(Lv).\]
We apply the Leibniz rule $[XY,Z]=X[Y,Z]+[X,Z]Y$ to see
\begin{align*}
    [L,A_1]&=[A_2^kA_1^l,A_1]\\
    &= A_2^k[A_1^l,A_1]+[A_2^k,A_1]A_1^l\\
    &=[A_2^k,A_1]A_1^l\\
    &=\sum_{i=1}^k A_2^{i-1}[A_2,A_1]A_2^{k-i}A_1^l\\
    &=-\sum_{i=1}^k A_2^{i-1}BA_2^{k-i}A_1^l.
\end{align*}
It follows that
\begin{align*}\phi(Lv)&=\operatorname{tr}([L,A_1]A_2)\\
&=-\sum_{i=1}^k \operatorname{tr}(A_2^{i-1}BA_2^{k-i}A_1^lA_2)\\
&=-\sum_{i=1}^k \operatorname{tr}(BA_2^{k-i}A_1^lA_2^i)\\
&=-\sum_{i=1}^k \phi(A_2^{k-i}A_1^lA_2^iv)\end{align*}
but by inverting inversions we have $A_2^{k-i}A_1^lA_2^iv=A_2^kA_1^lv=Lv$ proving
\[\phi(Lv)=-k\phi(Lv)\]
so that $\phi(Lv)=0$ as desired.
\end{proof}
This means that the commutativity locus of $\operatorname{ncNHilb}^{\underline{d}}(\A^2)$ is cut out by the commutator section on the quotient sheaf
\[\mathcal{H}om^{\operatorname{fil}}(\mathcal{V}_\bullet,\overline{\mathcal{V}}_{\bullet})\]
where $\overline{\mathcal{V}}_0=\operatorname{coker}(\mathfrak{v})$ for $\mathfrak{v}:\oo\to \mathcal{V}_0$ the universal cyclic vector. We quickly prove that this is in fact locally free; the main point is simply proving that $\overline{\mathcal{V}}_\bullet$ is a flag of bundles.
\begin{lemma}
    Consider the universal cyclic vector $\mathfrak{v}:\oo \to \mathcal{V}_0$. The cokernel $\overline{\mathcal{V}}_0$ is locally free and the subsheaves 
    \[\overline{\mathcal{V}}_i:=(\mathcal{V}_i+\operatorname{im}\mathfrak{v})/\operatorname{im}\mathfrak{v} \subset \overline{\mathcal{V}}_0\]
    are inclusions of locally free sheaves with locally free quotients.
\end{lemma}
\begin{proof}
 Let $\mathcal{L}:=\operatorname{im}\mathfrak{v}\subset \mathcal{V}_0$. By the cyclic condition the map $\mathfrak{v}$ has local constant rank $1$ so that $\mathcal{L}$ is a subbundle and so 
 \[\overline{\mathcal{V}}_0=\mathcal{V}_0/\mathcal{L}\]
 is locally free. If $\mathcal{V}_i=\mathcal{V}_0$ then $\overline{\mathcal{V}}_i=\overline{\mathcal{V}}_0$ is locally free. If $\mathcal{V}_i\subsetneqq \mathcal{V}_0$ then $\mathcal{L}\cap \mathcal{V}_i=0$ since if $\mathfrak{v}$ factors through $\mathcal{V}_i$ then $\mathcal{V}_i$ would be closed under the universal operators and contain the universal cyclic vector, hence by cyclicity $\mathcal{V}_i=\mathcal{V}_0$. It follows that 
 \[\mathcal{V}_i+\mathcal{L}\cong \mathcal{V}_i\oplus \mathcal{L}\]
 is locally free and 
 \[\overline{\mathcal{V}}_i\cong \mathcal{V}_i\]
 is also locally free.
\end{proof}

\subsection{The perfect obstruction theory}The above allows us to construct a virtual class on the nested Hilbert scheme of points on $\C^2$ with dimension matching the fundamental class in the smooth non-nested case. We note that equivariantly the commutativity section lives in 
\[\mathcal{H}om^{\operatorname{fil}}(\mathcal{V}_\bullet,t^{-1}_1t^{-1}_2\overline{\mathcal{V}}_{\bullet})=t^{-1}_1t^{-1}_2\mathcal{H}om^{\operatorname{fil}}(\mathcal{V}_\bullet,\overline{\mathcal{V}}_{\bullet})\]
where $t^{-1}_1t^{-1}_2\cdot $ means twisting by the corresponding character. Note that if $d_0\neq 0$ the rank of this bundle is given by 
\[h(\underline{d})-d_0\]
so that
\[\operatorname{dim}\operatorname{ncNHilb}^{\underline{d}}(\A^n)-\operatorname{rk}=h(\underline{d})+d-(h(\underline{d})-d_0)=d+d_0.\]
\begin{definition} \label{POT DEF} 
    Let $\underline{d}$ be a dimension vector. Define 
    \[\mathcal{E}=t^{-1}_1t^{-1}_2\mathcal{H}om^{\operatorname{fil}}(\mathcal{V}_\bullet,\overline{\mathcal{V}}_{\bullet})\]
    and
    \[\mathcal{E}_{\operatorname{naive}}=t^{-1}_1t^{-1}_2\mathcal{E}nd^{\operatorname{fil}}(\mathcal{V}_\bullet)\]
    as equivariant locally free sheaves on $\operatorname{ncNHilb}^{\underline{d}}(\A^2)$ and let $s_{\operatorname{com}}=[\Phi_1,\Phi_2]$ be the commutator section. Since the nested Hilbert scheme is its zero locus $Z(s_{\operatorname{com}})=\operatorname{NHilb}^{\underline{d}}(\A^2)$ (viewed either as a section of $\mathcal{E}$ and $\mathcal{E}'$) this defines two perfect obstruction theories of $\operatorname{NHilb}^{\underline{d}}(\A^2)$ of virtual dimensions $d$ and $d+d_i$ respectively, where $i$ is chosen to be the smallest such that $d_i\neq 0$. Concretely, if we let $I\subset \oo_{\operatorname{ncNHilb}^{\underline{d}}(\A^2)}$ they are given by restricting the map of complexes
    \begin{center}
        \begin{tikzcd}
            \big{[}\mathcal{E}_{\ast}^{\vee}\arrow{d}{s_{\operatorname{com}}^{\vee}} \arrow{r}{d\circ s_{\operatorname{com}}^{\vee}} & \Omega_{\operatorname{ncNHilb}^{\underline{d}}(\A^2)} \big{]}\arrow[equal]{d} \\ 
            \big{[}I/I^2\arrow{r}{d} &\Omega_{\operatorname{ncNHilb}^{\underline{d}}(\A^2)} \big{]}
        \end{tikzcd}
    \end{center}
    from $\operatorname{ncNHilb}^{\underline{d}}(\A^2)$ to $\operatorname{NHilb}^{\underline{d}}(\A^2)$ where $\ast\in \{\varnothing, \operatorname{naive}\}$. We let 
    \[\mathbb{E}\to \mathbb{L}_{\operatorname{NHilb}^{\underline{d}}(\A^2)}\]
    be the perfect obstruction theory defined by the bundle 
    \[\mathcal{E}_{\underline{d}}=\left\{\begin{array}{ll}
        \mathcal{E} & \text{if }d_0\neq0 \\
        \mathcal{E}_{\operatorname{naive}} & \text{if }d_0=0 
    \end{array}\right.\]
    so that the virtual dimension is $d+d_0$.
\end{definition}
As a first step, we identify $ds_{\operatorname{com}}$ with an explicit map of tautological bundles.
\begin{lemma} \label{differential identification lemma}
    We have a short exact sequence 
\[0\to  \mathcal{E}nd^{\operatorname{fil}}(\mathcal{V}_\bullet) \to t^{-1}_1\mathcal{E}nd^{\operatorname{fil}}(\mathcal{V}_\bullet)\oplus t^{-1}_2\mathcal{E}nd^{\operatorname{fil}}(\mathcal{V}_\bullet)\oplus \mathcal{V}_0\to T_{\operatorname{ncNHilb}^{\underline{d}}(\A^2)}\to 0 \]
of $T$-equivariant sheaves on $\operatorname{ncNHilb}^{\underline{d}}(\A^2)$. Using this, the map
    \[\text{d}s_{\operatorname{com}}:T_{\operatorname{ncNHilb}^{\underline{d}}(\A^2)}|_{\operatorname{NHilb}^{\underline{d}}(\A^2)}\to \mathcal{E}_{\operatorname{naive}}|_{\operatorname{NHilb}^{\underline{d}}(\A^2)}\]
    identifies with 
    \begin{center}
        \begin{tikzcd}
            (t^{-1}_1\mathcal{E}nd^{\operatorname{fil}}(\mathcal{V}_\bullet)\oplus t^{-1}_2\mathcal{E}nd^{\operatorname{fil}}(\mathcal{V}_\bullet)\oplus \mathcal{V}_0)\big{/}\mathcal{E}nd^{\operatorname{fil}}(\mathcal{V}_\bullet) \arrow{r} & t^{-1}_1t^{-1}_2\mathcal{E}nd^{\operatorname{fil}}(\mathcal{V}_\bullet)\\[-20pt]
            (\phi_1,\phi_2,v)\arrow[mapsto]{r} &\left[\phi_1,\Phi_2\right]+\left[\Phi_1,\phi_2\right]
        \end{tikzcd}
    \end{center}
    after restricting the latter to $\operatorname{NHilb}^{\underline{d}}(\A^2)$. The similar map with target $\mathcal{E}|_{\operatorname{NHilb}^{\underline{d}}(\A^2)}$ is obtained by simply composing the above with the quotient $\mathcal{E}_{\operatorname{naive}}\to \mathcal{E}$.
\end{lemma}
 \begin{proof}

We consider the Atiyah sequence 
\[0\to \mathfrak{p}\otimes_{\C} \mathcal{O}\to T_{\widetilde{\operatorname{ncNHilb}^{\underline{d}}}(\A^2)}\to \pi^*T_{\operatorname{ncNHilb}^{\underline{d}}(\A^2)}\to 0\]
for the principal $P_{\underline{d}}$-bundle $\pi:\widetilde{\operatorname{ncNHilb}^{\underline{d}}}(\A^2)\to \operatorname{ncNHilb}^{\underline{d}}(\A^2)$ where $\mathfrak{p}$ is the Lie algebra of $P_{\underline{d}}$. Setting $V_\bullet=N_\bullet\otimes_{\C} \oo_{\widetilde{\operatorname{ncNHilb}^{\underline{d}}}(\A^2)}$ for $N_\bullet$ the reference flag we see that
\[T_{\widetilde{\operatorname{ncNHilb}^{\underline{d}}}(\A^2)}=t^{-1}_1\mathcal{E}nd^{\operatorname{fil}}(V_\bullet)\oplus t^{-1}_2 \mathcal{E}nd^{\operatorname{fil}}(V_\bullet)\oplus V_0\]
while 
\[\mathfrak{p}\otimes_{\C} \mathcal{O}\cong \mathcal{E}nd^{\operatorname{fil}}(V_\bullet).\]
Note that $V_\bullet$ with its natural $P_{\underline{d}}$-equivariant structure descends to the universal flag $\mathcal{V}_\bullet$ under the equivalence
\[\operatorname{QCoh}(\widetilde{\operatorname{ncNHilb}^{\underline{d}}}(\A^2))^{P_{\underline{d}}}\cong \operatorname{QCoh}(\operatorname{ncNHilb}^{\underline{d}}(\A^2)).\]
Using faithful flatness of $\pi$, we therefore get that the Atiyah sequence descends to the desired short exact sequence
\[0\to  \mathcal{E}nd^{\operatorname{fil}}(\mathcal{V}_\bullet) \to t^{-1}_1\mathcal{E}nd^{\operatorname{fil}}(\mathcal{V}_\bullet)\oplus t^{-1}_2\mathcal{E}nd^{\operatorname{fil}}(\mathcal{V}_\bullet)\oplus \mathcal{V}_0\to T_{\operatorname{ncNHilb}^{\underline{d}}(\A^2)}\to 0 \]
where the first map is given by $\phi\mapsto ([\phi,\Phi_1],[\phi,\Phi_2],\phi \mathfrak{v})$. Similarly, on the commuting locus $Z(s_{\operatorname{com}})=\operatorname{NHilb}^{\underline{d}}(\A^2)$ the map
\[\text{d}s_{\operatorname{com}}:T_{\operatorname{ncNHilb}^{\underline{d}}(\A^2)}|_{\operatorname{NHilb}^{\underline{d}}(\A^2)}\to \mathcal{E}_{\operatorname{naive}}|_{\operatorname{NHilb}^{\underline{d}}(\A^2)}\]
of equivariant sheaves is induced by the map
\[t_1^{-1}\mathcal{E}nd^{\operatorname{fil}}(V_\bullet)\oplus t_2^{-1}\mathcal{E}nd^{\operatorname{fil}}(V_\bullet)\to t_1^{-1}t_2^{-1}\mathcal{E}nd^{\operatorname{fil}}(V_\bullet)\]
given by $(\phi_1,\phi_2,v)\mapsto [\phi_1,\widetilde{\Phi}_2]+[\widetilde{\Phi}_1,\phi_2]$ which descends to the map of the lemma.
\end{proof}

As a first step, we consider the non-nested case $\underline{d}=(d)$. Here, we simply have the Hilbert scheme of points $\operatorname{Hilb}^d(\A^2)$ and we will prove that $\mathbb{E}\to \mathbb{L}_{\operatorname{Hilb}^d(\A^2)}$ is indeed just the smooth perfect obstruction theory. We similarly have that $\operatorname{NHilb}^{(0,d)}(\A^n)=\operatorname{Hilb}^d(\A^n)$ and we prove that $\mathbb{E}\to \mathbb{L}_{\operatorname{NHilb}^{(0,d)}(\A^2)}$ has obstruction bundle $h^{-1}(\mathbb{E})^\vee$ given by the dual of the tautological bundle of the canonical line bundle of $\C^2$.
\begin{lemma}
    Let $d\in \Z$. The map $\mathbb{E}\to \mathbb{L}_{\operatorname{Hilb}^d(\A^2)}$ is an equivalence, i.e. the perfect obstruction theory is the smooth one. \label{smooth-obstruction}
\end{lemma}
\begin{proof}
By smoothness of $\operatorname{Hilb}^d(\A^2)$ we have $\mathbb{L}_{\operatorname{Hilb}^d(\A^2)}=\Omega_{\operatorname{Hilb}^d(\A^2)}[0]$. Since we also have that the virtual dimension agrees with the actual dimension it will suffice to prove that
    \[\mathcal{E}^\vee \to \Omega_{\operatorname{ncHilb}^d(\A^2)}\]
is injective. Using the trace pairing and \cref{differential identification lemma} we can explicitly identify this map with
\[\mathcal{H}om(\overline{\mathcal{V}},\mathcal{V})\to \operatorname{ker}\left(\mathcal{E}nd(\mathcal{V})^{\oplus2}\oplus \mathcal{V}\to \mathcal{E}nd(\mathcal{V})\right)\]
which is given by $\phi\mapsto ([\Phi_2,\phi q],[\phi q,\Phi_1],0)$ for $q:\mathcal{V}\to \overline{\mathcal{V}}$ the quotient map. If $\phi$ is in the kernel of this map, then $\phi q$ is an endomorphism of $\mathcal{V}$ which vanishes on the cyclic vector and commutes with $\Phi_1,\Phi_2$ hence it must be $0$. Since $q$ is an epimorphism we conclude that $\phi=0$ as desired.  
\end{proof}
\begin{lemma}
    Let $d\in \Z$. For $\mathbb{E}\to \mathbb{L}_{\operatorname{NHilb}^{(0,d)}(\A^2)}$ we have $h^{-1}(\mathbb{E})^\vee\cong (\omega_{\A^2}^{[d]})^\vee$ i.e. the obstruction bundle is the dual of the tautological bundle of the dualizing bundle.
\end{lemma}
\begin{proof}
    It will suffice to identify
    \[ t^{-1}_1t^{-1}_2\mathcal{V}|_{\operatorname{Hilb}^d(\A^2)} \cong \operatorname{ker} (\mathcal{E}_{\operatorname{naive}}^{\vee}\to \Omega_{\operatorname{ncHilb}^d(\A^2)})|_{\operatorname{Hilb}^d(\A^2)}\]
    since 
    \[t^{-1}_1t^{-1}_2\mathcal{V}|_{\operatorname{Hilb}^d(\A^2)}=t^{-1}_1t^{-1}_2\oo_{\A^2}^{[d]}=(t^{-1}_1t^{-1}_2\oo_{\A^2})^{[d]}=\omega_{\A^2}^{[d]}\]
    by construction. Repeating the arguments in the proof of \cref{smooth-obstruction} we can directly identify the above kernel with the operators of
    \[t^{-1}_1t^{-1}_2\mathcal{E}nd(\mathcal{V})\cong \mathcal{E}_{\operatorname{naive}}^\vee\]
    which commutes with the universal operators $\Phi_1,\Phi_2$. Clearly such an operator is uniquely determined by its value on the cyclic vector by the cyclicity condition. Conversely, on the commutativity locus we have an algebra structure on $\mathcal{E}$ viewing it as a quotient of $\oo[x,y]$ by a left (hence two-sided) ideal. It follows that for any $w\in \mathcal{V}$ we have an operator $L_w\in \mathcal{E}nd(\mathcal{V})$ given by left multiplication with $w$. This commutes with $\Phi_1,\Phi_2$ and evaluates to $w$ on the cyclic vector. We conclude that the evaluation map 
    \[\operatorname{ev}_{\mathfrak{v}}:t_1^{-1}t_2^{-1}\mathcal{E}nd(\mathcal{V})|_{\operatorname{Hilb}^{d}(\A^2)}\to t_1^{-1}t_2^{-1}\mathcal{V}|_{\operatorname{Hilb}^d(\A^2)}\]
    restricts to an isomorphism on the above kernel as desired.
\end{proof}

Next, we will prove that our virtual class is equal to that of Gholampour, Sheshmani and Yau \cite{GHOLAMPOUR2020107046}. 
\begin{prop}\label{GSY comparison}
    Let $\underline{d}=(d_0,\dots, d_{r})\in \Z^{r+1}_{\geq 0}$. Let $\A^2\subset \mathbb{P}^2$ as the complement of the hyperplane at $\infty$. Let $\widetilde{\mathbb{E}}_{\operatorname{GSY}}\to \mathbb{L}_{\operatorname{NHilb}^{\underline{d}}(\mathbb{P}^2)}$ be the perfect obstruction theory defined in \cite{GHOLAMPOUR2020107046} for the Hilbert scheme of a smooth projective surface. Let $\mathbb{E}_{\operatorname{GSY}}\to \mathbb{L}_{\operatorname{NHilb}^{\underline{d}}(\mathbb{P}^2)}$ be the restriction to the Zariski open $\operatorname{NHilb}^{\underline{d}}(\A^2)\hookrightarrow \operatorname{NHilb}^{\underline{d}}(\mathbb{P}^2)$ induced by $\A^2\hookrightarrow\mathbb{P}^2$. Then in $T$-equivariant $K$-theory we have the identity
    \[[\mathbb{E}_{\operatorname{GSY}}]=[\mathbb{E}]\]
    where $\mathbb{E}$ is the perfect obstruction theory constructed above \cref{POT DEF}. In particular, both theories induce the same virtual invariants, i.e. the same virtual fundamental class and virtual structure sheaf agree.
\end{prop}
\begin{proof}
    We dualize and show that virtual tangent bundles agree. To avoid cumbersome notation we will for $\mathcal{F}$ a coherent/locally free sheaf simply denote its $K$-theory class by $\mathcal{F}$ instead of $[\mathcal{F}]$. By definition
    \[\mathbb{E}^{\vee}=\left\{\begin{array}{cl}
      T_{\operatorname{ncNHilb}^{\underline{d}}(\A^2)}|_{\operatorname{NHilb}^{\underline{d}}(\A^2)}-\mathcal{E}|_{\operatorname{NHilb}^{\underline{d}}(\A^2)}   & \text{if } d_0\neq 0 \\
      T_{\operatorname{ncNHilb}^{\underline{d}}(\A^2)}|_{\operatorname{NHilb}^{\underline{d}}(\A^2)}-\mathcal{E}_{\operatorname{naive}}|_{\operatorname{NHilb}^{\underline{d}}(\A^2)}   & \text{if } d_0= 0
    \end{array}\right.\]
    By \cref{differential identification lemma} we can rewrite $T_{\operatorname{ncNHilb}^{\underline{d}}(\A^2)}|_{\operatorname{NHilb}^{\underline{d}}(\A^2)}$ as
    \[\mathcal{V}_0+(t^{-1}_1+t^{-1}_2-1)\mathcal{E}nd^{\operatorname{fil}}(\mathcal{V}_\bullet)=\mathcal{V}_0+(t^{-1}_1+t^{-1}_2-1)\sum_{0\leq i \leq j\leq r-1} (\operatorname{gr}^i\mathcal{V}_\bullet)^\vee \cdot \operatorname{gr}^j\mathcal{V}_\bullet.\]
    Letting $\mathcal{Q}_i=\oo_{\A^2}^{[d_0,\dots, d_{i-1}]}$ be the tautological sheaf for the $i$'th subscheme of the nesting, we have for each $0<i\leq r+1$ a short exact sequence
    \[0\to \mathcal{V}_i\to \mathcal{Q}_{r+1}\to \mathcal{Q}_i\to 0\]
    and if we set $\mathcal{Q}_0=0$ by convention then the above extends to $i=0$. In particular, we get
    \[\operatorname{gr}^i(\mathcal{V}_\bullet)=\mathcal{V}_i-\mathcal{V}_{i+1}=\mathcal{Q}_{i+1}-\mathcal{Q}_i.\]
    We can then the rewrite 
    \begin{align*}
        \sum_{0\leq i \leq j\leq r} (\operatorname{gr}^i\mathcal{V}_\bullet)^\vee \cdot \operatorname{gr}^j\mathcal{V}_\bullet&= \sum_{0\leq i \leq j\leq r} (\mathcal{Q}_{i+1}^{\vee}-\mathcal{Q}_{i}^{\vee})(\mathcal{Q}_{j+1}-\mathcal{Q}_{j}) \\
        &=\sum_{j=0}^{r}(\mathcal{Q}_{j+1}-\mathcal{Q}_{j})\sum_{i=0}^j(\mathcal{Q}_{i+1}^{\vee}-\mathcal{Q}_{i}^{\vee}) \\
        &=\sum_{j=0}^{r}\mathcal{Q}_{j+1}^{\vee}(\mathcal{Q}_{j+1}-\mathcal{Q}_{j})
    \end{align*}
where the third equality uses that the sum is alternating. It follows that
\[T_{\operatorname{ncNHilb}^{\underline{d}}(\A^2)}|_{\operatorname{NHilb}^{\underline{d}}(\A^2)}=\mathcal{Q}_{r+1}+(t^{-1}_1+t^{-1}_1-1)\sum_{j=0}^{r}\mathcal{Q}_{j+1}^{\vee}(\mathcal{Q}_{j+1}-\mathcal{Q}_{j}).\]
If $d_0\neq 0$ we have
\[\mathcal{E}|_{\operatorname{NHilb}^{\underline{d}}(\A^2)}=t^{-1}_1t^{-1}_2\mathcal{H}om^{\operatorname{fil}}(\mathcal{V}_\bullet, \overline{\mathcal{V}}_\bullet)=t^{-1}_1t^{-1}_2\sum_{0\leq i\leq j \leq r} (\operatorname{gr}^i\mathcal{V}_\bullet)^\vee(\operatorname{gr}^j\overline{\mathcal{V}_\bullet})\]
where 
\[\overline{\mathcal{V}}_i=\left\{\begin{array}{cc}
    \mathcal{V}_0-1 & i=0 \\
    \mathcal{V}_i & i>0
\end{array}\right.\]
hence 
\[\mathcal{E}|_{\operatorname{NHilb}^{\underline{d}}(\A^2)}=t_1^{-1}t_2^{-1}\left(-\operatorname{gr^0}\mathcal{V}_\bullet+\sum_{0\leq i\leq j \leq r} (\operatorname{gr}^i\mathcal{V}_\bullet)^\vee(\operatorname{gr}^j\mathcal{V}_\bullet)\right).\]
For $d_0=0$ we have 
\[\mathcal{E}_{\operatorname{naive}}|_{\operatorname{NHilb}^{\underline{d}}(\A^2)}=t^{-1}_1t^{-1}_2\mathcal{H}om^{\operatorname{fil}}(\mathcal{V}_\bullet, \mathcal{V}_\bullet)=t^{-1}_1t^{-1}_2\sum_{0\leq i\leq j \leq r} (\operatorname{gr}^i\mathcal{V}_\bullet)^\vee(\operatorname{gr}^j\mathcal{V}_\bullet)\]
and so using that $\operatorname{gr}^0(\mathcal{V}_\bullet)=0$ we get the same identity 
\[\mathcal{E}_{\operatorname{naive}}|_{\operatorname{NHilb}^{\underline{d}}(\A^2)}=t_1^{-1}t_2^{-1}\left(-\operatorname{gr^0}\mathcal{V}_\bullet+\sum_{0\leq i\leq j \leq r} (\operatorname{gr}^i\mathcal{V}_\bullet)^\vee(\operatorname{gr}^j\mathcal{V}_\bullet)\right).\]
Using the above we can then identify the bundle term in either case with
\[-t^{-1}_1t^{-1}_2\mathcal{Q}_1^\vee+t^{-1}_1t^{-1}_2\sum_{j=0}^{r}\mathcal{Q}_{j+1}^{\vee}(\mathcal{Q}_{j+1}-\mathcal{Q}_{j}).\]
Using the identity 
\[t^{-1}_1+t^{-1}_2-1-t^{-1}_1t^{-1}_2=-\frac{(1-t_1)(1-t_2)}{t_1t_2}\]
and the above we now arrive at the formula
\[\mathbb{E}^{\vee}=\mathcal{Q}_{r+1}+\frac{\mathcal{Q}_1^\vee}{t_1t_2}-\frac{(1-t_1)(1-t_2)}{t_1t_2}\sum_{j=0}^{r}\mathcal{Q}_{j+1}^{\vee}(\mathcal{Q}_{j+1}-\mathcal{Q}_{j}).\]
On the other side, we use the virtual tangent space found in \cite[Theorem 2]{Gholampour2017DegeneracyLV}. Let $n_i=d_0+\cdots +d_{i-1}$ for $i=1,\dots, r+1$. Then 
\[\widetilde{\mathbb{E}}_{\operatorname{GSY}}^{\vee}\cong \left[ T_{\operatorname{Hilb}^{n_{r+1}}(\mathbb{P}^2)}\oplus \cdots \oplus \operatorname{T}_{\operatorname{Hilb}^{n_1}(\mathbb{P}^2)}\to \mathcal{E}xt^1_{\widetilde{p}}(\widetilde{\mathcal{I}}_{r+1},\widetilde{\mathcal{I}}_{r})_0\oplus \cdots \oplus \mathcal{E}xt_{\widetilde{p}}^1(\widetilde{\mathcal{I}}_2,\widetilde{\mathcal{I}}_1)_0\right]\]
where 
\[\widetilde{p}:\operatorname{NHilb}^{\underline{d}}(\mathbb{P}^2)\times \mathbb{P}^2\to \operatorname{NHilb}^{\underline{d}}(\mathbb{P}^2)\]
is the projection, $\widetilde{\mathcal{I}}_i$ are the ideal sheaves of the universal subschemes $\widetilde{\mathcal{Z}}_i\subset \operatorname{NHilb}^{\underline{d}}(\mathbb{P}^2)\times \mathbb{P}^2$ and $\mathcal{E}xt_{\widetilde{p}}^i(-,-)$ is the $i$'th cohomology sheaf of 

\[R\mathcal{H}om_{\widetilde{p}}(-,-):=\widetilde{p}_*R\mathcal{H}om(-,-)\]
with the subscript $0$ denoting the trace free part. Let $j:\A^2\hookrightarrow \mathbb{P}^2$ and $\iota: \operatorname{NHilb}^{\underline{d}}(\A^2)\hookrightarrow \operatorname{NHilb}^{\underline{d}}(\mathbb{P}^2)$ be the open immersions. Clearly $\iota^* T_{\operatorname{Hilb}^{n_i}(\mathbb{P}^2)}\cong T_{\operatorname{Hilb}^{n_i}(\mathbb{C}^2)}$ and since our perfect obstruction theory for $r=1$ is the smooth one the above calculations show that
\[\iota^* T_{\operatorname{Hilb}^{n_i}(\mathbb{P}^2)}=\mathcal{Q}_i+\frac{\mathcal{Q}_i^\vee}{t_1t_2}-\frac{(1-t_1)(1-t_2)}{t_1t_2}\mathcal{Q}_{i}^\vee \mathcal{Q}_i.\]
To make the notation less cumbersome for the calculation of the ext-term, we will set $S=\mathbb{A}^2$, $\widetilde{S}=\mathbb{P}^2$, $X=\operatorname{NHilb}^{\underline{d}}(\A^2)$ and $\widetilde{X}=\operatorname{NHilb}^{\underline{d}}(\mathbb{P}^2)$. We will let $\mathcal{Z}_k\subset X\times S$ be the universal families with ideals $\mathcal{I}_k$. Note that by definition $(\iota\times \id)^* \widetilde{\mathcal{Z}}_k\subset X\times \widetilde{S}$ is supported on $X\times S$ and agrees with $\mathcal{Z}_k$. We finally let $\omega_S$ and $\omega_{\widetilde{S}}$ be the dualizing bundles. We consider the diagram
\begin{center}
    \begin{tikzcd}
        X\times S \arrow{r}{\operatorname{id}\times j} \arrow{rd}{p} & X\times \widetilde{S} \arrow{d}{p'} \arrow{r}{\iota \times \id} & \widetilde{X}\times \widetilde{S} \arrow{d}{\widetilde{p}} \\
        & X \arrow{r}{\iota} & \widetilde{X}
    \end{tikzcd}
\end{center}
where the square is cartesian. Using that 
\[\widetilde{\mathcal{I}}_k=\oo_{\widetilde{X}\times \widetilde{S}}-\mathcal{O}_{\widetilde{\mathcal{Z}}_k}\]
in $K$-theory we get that $R\mathcal{H}om_{\widetilde{p}}(\widetilde{\mathcal{I}}_{i+1},\widetilde{\mathcal{I}}_{i})$ equals
\[R\mathcal{H}om_{\widetilde{p}}(\oo_{\widetilde{X}\times \widetilde{S}},\oo_{\widetilde{X}\times \widetilde{S}})-R\mathcal{H}om_{\widetilde{p}}(\oo_{\widetilde{X}\times \widetilde{S}},\oo_{\widetilde{\mathcal{Z}_{i}}})-R\mathcal{H}om_{\widetilde{p}}(\oo_{\widetilde{\mathcal{Z}}_{i+1}},\oo_{\widetilde{X}\times \widetilde{S}})+R\mathcal{H}om_{\widetilde{p}}(\oo_{\widetilde{\mathcal{Z}}_{i+1}},\oo_{\widetilde{\mathcal{Z}_{i}}})\]
in $K$-theory. We see that
\[R\mathcal{H}om_{\widetilde{p}}(\oo_{\widetilde{X}\times \widetilde{S}},\oo_{\widetilde{X}\times \widetilde{S}})=R\widetilde{p}_* \oo_{\widetilde{X}\times \widetilde{S}}=\oo_{\widetilde{X}}\]
by the usual calculation of the cohomology of $\mathcal{O}_{\mathbb{P}^2}$. Similarly, we have 
\[R\mathcal{H}om_{\widetilde{p}}(\oo_{\widetilde{X}\times \widetilde{S}},\oo_{\widetilde{\mathcal{Z}_{i}}})=R\widetilde{p}_*\oo_{\widetilde{\mathcal{Z}}_i}.\]
and by Grothendieck duality we have quasi-isomorphisms
\[R\mathcal{H}om_{\widetilde{p}}(\oo_{\widetilde{\mathcal{Z}}_{i+1}},\oo_{\widetilde{X}\times \widetilde{S}})[2]\cong R\mathcal{H}om_{\widetilde{p}}(\omega_{\widetilde{S}}^{-1},\mathcal{O}_{\widetilde{\mathcal{Z}}_{i+1}})^\vee\cong (R\widetilde{p}_*(\omega_{\widetilde{S}}\otimes\oo_{\widetilde{\mathcal{Z}}_{i+1}}))^\vee\]
so that in $K$-theory we get
\[R\mathcal{H}om_{\widetilde{p}}(\oo_{\widetilde{\mathcal{Z}}_{i+1}},\oo_{\widetilde{X}\times \widetilde{S}})=(R\widetilde{p}_*(\omega_{\widetilde{S}}\otimes\oo_{\mathcal{Z}_{i+1}}))^\vee\]
as the shift $[2]$ contributes with the sign $(-1)^2$. Let
\[R\mathcal{H}om_{\widetilde{p}}(\widetilde{\mathcal{I}}_{i+1},\widetilde{\mathcal{I}}_{i})=\mathcal{E}xt^0_{\widetilde{p}}(\widetilde{\mathcal{I}}_{i+1},\widetilde{\mathcal{I}}_{i})-\mathcal{E}xt^1_{\widetilde{p}}(\widetilde{\mathcal{I}}_{i+1},\widetilde{\mathcal{I}}_{i})+\mathcal{E}xt^2_{\widetilde{p}}(\widetilde{\mathcal{I}}_{i+1},\widetilde{\mathcal{I}}_{i}).\]
We first note that
\[\mathcal{H}om_{\widetilde{p}}(\widetilde{\mathcal{I}}_{i+1},\widetilde{\mathcal{I}}_{i})\hookrightarrow\mathcal{H}om_{\widetilde{p}}(\widetilde{\mathcal{I}}_{i+1},\mathcal{O}_{\widetilde{X}\times \widetilde{S}})\cong R\widetilde{p}_*\mathcal{O}_{\widetilde{X}\times \widetilde{S}}\]
and the inclusion $\widetilde{\mathcal{I}}_{i+1}\subset \widetilde{\mathcal{I}}_i$ gives us an inverse providing us with an identification
\[\mathcal{E}xt^0_{\widetilde{p}}(\widetilde{\mathcal{I}}_{i+1},\widetilde{\mathcal{I}}_{i})=R\widetilde{p}_*\oo_{\widetilde{X}\times \widetilde{S}}=\oo_{\widetilde{X}}.\]
For a closed point corresponding to $I_{i+1}\subset I_i$ we have 
\[\operatorname{Hom}_{\widetilde{S}}(I_{i},I_{i+1}\otimes_{\widetilde{S}}\omega_{\widetilde{S}})\hookrightarrow \operatorname{Hom}_{\widetilde{S}}(I_i,\omega_{\widetilde{S}})=H^0(\widetilde{S},\omega_{\widetilde{S}})=0\]
where the last calculation follows since $\omega_{\widetilde{S}}=\mathcal{O}(-3)$. We conclude by Serre duality that $\operatorname{Ext}^2_{\widetilde{S}}(I_{i+1},I_{i})=0$ and hence by Nakayama that 
\[\mathcal{E}xt^2_{\widetilde{p}}(\widetilde{\mathcal{I}}_{i+1},\widetilde{\mathcal{I}}_{i})=0.\]
Note also that
\[\mathcal{E}xt^1_{\widetilde{p}}(\widetilde{\mathcal{I}}_{i+1},\widetilde{\mathcal{I}}_{i})_0=\mathcal{E}xt^1_{\widetilde{p}}(\widetilde{\mathcal{I}}_{i+1},\widetilde{\mathcal{I}}_{i})\]
since $R^1\widetilde{p}_*\oo_{\widetilde{X}\times \widetilde{S}}=0$. We conclude that $\iota^*\mathcal{E}xt^1_{\widetilde{p}}(\widetilde{\mathcal{I}}_{i+1},\widetilde{\mathcal{I}}_{i})_0$ equals 
\[\iota^*R\widetilde{p}_* \mathcal{O}_{\widetilde{\mathcal{Z}_i}}+\iota^*(R\widetilde{p}_*(\omega_{\widetilde{S}}\otimes\oo_{\widetilde{\mathcal{Z}}_{i+1}}))^\vee-\iota^*R\mathcal{H}om_{\widetilde{p}}(\oo_{\widetilde{\mathcal{Z}}_{i+1}},\oo_{\widetilde{\mathcal{Z}_{i}}}).\]
By cohomology and base change we have $\iota^*R\widetilde{p}_*=Rp'_*(\iota\times \id)^*$ and using that $(\iota \times \id)^*\mathcal{O}_{\widetilde{\mathcal{Z}}_k}=(\id \times j)_*\oo_{\mathcal{Z}_k}$ we see that the above equals 
\[Rp_* \mathcal{O}_{\mathcal{Z}_i}+(Rp_*(\omega_{S}\otimes\oo_{\mathcal{Z}_{i+1}}))^\vee -R\mathcal{H}om_{p}(\oo_{\mathcal{Z}_{i+1}},\oo_{\mathcal{Z}_{i}}).\]
Note that $p$ is affine with $p_*\mathcal{O}_{\mathcal{Z}_k}=\mathcal{Q}_k$ by definition. It follows that
\[Rp_*\mathcal{O}_{\mathcal{Z}_i}=\mathcal{Q}_i\]
and since $\omega_S=p^*(t_1t_2\mathcal{O}_X)$ we also get
\[(Rp_*(\omega_{S}\otimes\oo_{\mathcal{Z}_{i+1}}))^\vee=(t_1t_2\mathcal{Q}_{i+1})^\vee=\frac{\mathcal{Q}_{i+1}^{\vee}}{t_1t_2}.\]
Since $\mathcal{O}_{\mathcal{Z}_{i+1}}$ is a perfect complex and $X\times S$ has the resolution property we can find a finite resolution of $T$-equivariant bundles, which is necessarily of the form
\[E_m\otimes_X \oo_{X\times S}\to E_{m-1}\otimes_X \oo_{X\times S}\to \cdots \to E_1\otimes_X \oo_{X\times S}\to \oo_{\mathcal{Z}_{i+1}}\]
for $E_i$ locally free T-equivariant sheaves on $X$ by $K$-theoretic homotopy invariance since $S=\A^2$. Then in $K$-theory
\[R\mathcal{H}om_{p}(\oo_{\mathcal{Z}_{i+1}},\oo_{\mathcal{Z}_{i}})=\sum_{j=1}^m(-1)^jRp_*R\mathcal{H}om_{X\times S}(Lp^*E_j,\mathcal{O}_{\mathcal{Z}_i})=\sum_{j=1}^{m}(-1)^j\mathcal{R}\mathcal{H}om_{X}(E_i,Rp_*\oo_{\mathcal{Z}_i})\]
so that 
\[R\mathcal{H}om_{p}(\oo_{\mathcal{Z}_{i+1}},\oo_{\mathcal{Z}_{i}})=\sum_{j=1}^m(-1)^jE_j^\vee \mathcal{Q}_i.\]
But we see that
\[\mathcal{Q}_{i+1}=p_*\oo_{\mathcal{Z}_{i+1}}=\sum_{j=1}^m (-1)^jp_*(E_i\otimes_X \oo_{X\times S}) \]
and as an $\oo_X$ algebra we have $\oo_{X\times S}=\oo_X[x_1,x_2]$ which has $T$-character
\[\sum_{(a,b)\in \Z^2} t_1^{a}t_2^{b}\oo_X=\frac{1}{(1-t_1)(1-t_2)}\mathcal{O}_X\]
so that 
\[\mathcal{Q}_{i+1}=\sum(-1)^j\frac{1}{(1-t_1)(1-t_2)}E_j\]
hence 
\[\sum(-1)^jE_j^\vee=(1-t_1^{-1})(1-t_2^{-1})\mathcal{Q}_{i+1}^{\vee}=\frac{(1-t_1)(1-t_2)}{t_1t_2}\mathcal{Q}_{i+1}^\vee.\]
Combining this we see that
\[\iota^*\mathcal{E}xt_{\widetilde{p}}^1(\widetilde{\mathcal{I}}_i,\widetilde{\mathcal{I}}_{i+1})_0=\mathcal{Q}_i+\frac{\mathcal{Q}_{i+1}^{\vee}}{t_1t_2}-\frac{(1-t_1)(1-t_2)}{t_1t_2}\mathcal{Q}_{i+1}^\vee\mathcal{Q}_i\]
and so $\mathbb{E}_{\operatorname{GSY}}^\vee$ equals
\begin{align*}
    &\sum_{i=1}^{r+1}\left(\mathcal{Q}_i+\frac{\mathcal{Q}_i^\vee}{t_1t_2}-\frac{(1-t_1)(1-t_2)}{t_1t_2}\mathcal{Q}_{i}^\vee \mathcal{Q}_i\right)-\sum_{i=1}^{r}\left(\mathcal{Q}_i+\frac{\mathcal{Q}_{i+1}^{\vee}}{t_1t_2}-\frac{(1-t_1)(1-t_2)}{t_1t_2}\mathcal{Q}_{i+1}^\vee\mathcal{Q}_i \right)\\
    &=\mathcal{Q}_{r+1}+\frac{\mathcal{Q}_1^\vee}{t_1t_2}-\frac{(1-t_1)(1-t_2)}{t_1t_2}\sum_{j=0}^{r}\mathcal{Q}_{j+1}^{\vee}(\mathcal{Q}_{j+1}-\mathcal{Q}_{j})
\end{align*}
as desired. 
\end{proof}
We end this section by proving that two maps between nested Hilbert schemes forgetting the largest or the smallest subscheme may be equipped with relative perfect obstruction theories compatible with the obstruction theories above. Recall that for the dimension vector $(d_0,\dots, d_r)$ we keep track of a nesting of $r+1$ ideals 
\[I_{r+1}\subset \cdots \subset I_1\]
in $I_0=\C[x,y]$ via the flag of vectorspaces $V_\bullet$ with
\[V_i=I_i/I_{r+1}.\]
The map forgetting the first ideal simply forgets the $V_1$ part of the flag and changes the dimension vector to 
\[(d_0+d_1,d_2,\dots,d_r).\]
The map forgetting the last, i.e. $(r+1)$'th ideal takes $V_\bullet$ to the quotient flag $W_\bullet$ where 
\[W_i=V_i/V_r\cong I_i/I_r\]
and changes the dimension vector to 
\[(d_0,\dots, d_{r-1}).\]
We start by showing the case of forgetting the smallest subscheme, i.e. the first ideal.
\begin{lemma} \label{forget small subscheme}
Let $\underline{d}=(d_0,\dots, d_r)\in \Z_{\geq 0}^{r+1}$ be a dimension vector. Set 
\[\underline{d}^{(1)}=(d_0+d_1,d_2,\dots, d_r)\in \Z_{\geq 0}^{r}\]
and let
\[f=f_{\underline{d},\underline{d}^{(1)}}:\operatorname{ncNHilb}^{\underline{d}}(\A^2)\to \operatorname{nc
NHilb}^{\underline{d}^{(1)}}(\A^2) \]
be the map which is given by the non-commutative tuple $(\mathcal{V}_{\bullet}^{\hat{1}},\Phi_1,\Phi_2,\mathfrak{v})$ where 
\[\mathcal{V}_{\bullet}^{\hat{1}}=(\mathcal{V}_0\supset \mathcal{V}_2\supset\mathcal{V}_3  \cdots \supset \mathcal{V}_{r+1}=0)\]
i.e., where we simply forget the first part of the filtration. There exists a map of equivariant bundles
\[\psi:\mathcal{E}_{\underline{d}}\to f^*\mathcal{E}_{\underline{d}^{(1)}}\]
such that the pair $(f,\psi)$ satisfies the conditions of \cref{relative POT criterion}. In particular, the corresponding forgetful map 
\[\operatorname{NHilb}^{\underline{d}}(\A^2)\to \operatorname{NHilb}^{\underline{d}^{(1)}}(\A^2)\]
can be equipped with a relative perfect obstruction theory compatible with the obstruction theories defined in \cref{POT DEF}.
\end{lemma}
\begin{proof}
    Suppose first that $d_0\neq 0$. Define $\psi$ as the inclusion 
    \[t^{-1}_1t^{-1}_2\mathcal{H}om^{\operatorname{fil}}(\mathcal{V}_{\bullet},\overline{\mathcal{V}}_\bullet)\hookrightarrow t^{-1}_1t^{-1}_2\mathcal{H}om^{\operatorname{fil}}(\mathcal{V}^{\hat{1}}_{\bullet},\overline{\mathcal{V}}^{\hat{1}}_\bullet).\]
    Clearly $\psi$ maps the commutator of the universal maps to itself. It remains to show the transversality condition. Using faithful flatness of the group quotient and injectivity of the map
    \[\pi^*\Omega_{\operatorname{ncNHilb}^{\underline{d}}(\A^2)}\to \Omega_{\widetilde{\operatorname{ncNHilb}^{\underline{d}}}(\A^2)} \]
    we can reduce the calculation to $\widetilde{\operatorname{ncNHilb}^{\underline{d}}}(\A^2)$. If we set $V_\bullet=N_\bullet \otimes \oo_{\widetilde{\operatorname{ncNHilb}^{\underline{d}}}(\A^2)}$ then 
    \[\pi^*\mathcal{E}_{\underline{d}}=\mathcal{H}om^{\operatorname{fil}}(V_\bullet, \overline{V}_\bullet),\]
    \[\pi^*f^*\mathcal{E}_{\underline{d}^{(1)}}=\mathcal{H}om^{\operatorname{fil}}(V^{\hat{1}}_\bullet, \overline{V}_\bullet^{\hat{1}})\]
    and
    \[f^*T_{\widetilde{\operatorname{ncNHilb}^{\underline{d}^{(1)}}}(\A^2)}=\mathcal{E}nd^{\operatorname{fil}}(V_\bullet^{\hat{1}})^{\oplus 2}\oplus V_0.\]
    By dualizing we should show that the map
    \[(\mathcal{E}nd^{\operatorname{fil}}(V_\bullet^{\hat{1}})^{\oplus 2}\oplus V_0) \oplus\mathcal{H}om^{\operatorname{fil}}(V_\bullet, \overline{V}_\bullet) \xrightarrow{(-\text{d}s_{\operatorname{com}},\phi)} \mathcal{H}om^{\operatorname{fil}}(V_\bullet^{\hat{1}}, \overline{V}_\bullet^{\hat{1}})\]
is surjective on  the commutativity locus. Equivalently, by taking the quotient by the second summand of the source which is an inclusion we should show that the map
\[\mathcal{E}nd^{\operatorname{fil}}(V_\bullet^{\hat{1}})^{\oplus 2}\oplus V_0\xrightarrow{-\text{d}s_{\operatorname{com}}}\mathcal{H}om^{\operatorname{fil}}(V_\bullet^{\hat{1}}, \overline{V}_\bullet^{\hat{1}})\big{/}\mathcal{H}om^{\operatorname{fil}}(V_\bullet, \overline{V}_\bullet)=\mathcal{H}om(V_1/V_{2},\overline{V}_{0}/\overline{V}_1)\]
is a surjection. We can remove the $V_0$ summand from the source, since $\text{d}s_{\operatorname{com}}$ is zero on it. Taking the dual and using the trace pairing we should show injectivity of 
\[\mathcal{H}om(\overline{V}_{0}/\overline{V}_{1},V_1/V_{2})\to( \mathcal{E}nd(V_\bullet^{\hat{1}})/\mathcal{H}om^{\operatorname{fil}}(V^{\hat{1}}_{\bullet},V_{\bullet+1}^{\hat{1}}))^{\oplus 2}.\]
If $\varphi:\overline{V}_{0}/\overline{V}_{1}\to V_1/V_{2}$ is in the kernel of this map then for $k=1,2$ the map 
\[[\widetilde{\Phi}_k,\varphi]\]
must take $V^{\hat{1}}_{0}=V_{0}$ to $V^{\hat{1}}_1=V_{2}$ and therefore vanish as a map $V_0/V_1\to V_1/V_2$. But this means that $\varphi$ commutes with $\widetilde{\Phi}_k$ and vanishes on the cyclic vector so that it must be 0. This finishes the proof in the $d_0\neq 0$ case. In the case where $d_0=0$ we have $\mathcal{V}_0=\mathcal{V}_1$ so the flag-preserving maps of $\mathcal{V}_\bullet$ and $\mathcal{V}_{\bullet}^{\hat{1}}$ are the same. When $d_1= 0$ we can simply take the identity map
\[\psi:t_1^{-1}t_2^{-1}\mathcal{H}om^{\operatorname{fil}}(\mathcal{V}_\bullet,\mathcal{V}_\bullet)= t_1^{-1}t_2^{-1}\mathcal{H}om^{\operatorname{fil}}(\mathcal{V}_\bullet^{\hat{1}},\mathcal{V}_\bullet^{\hat{1}})\]
while in the case where $d_1\neq 0$ we can take the quotient map
\[\psi:t_1^{-1}t_2^{-1}\mathcal{H}om^{\operatorname{fil}}(\mathcal{V}_\bullet,\mathcal{V}_\bullet)= t_1^{-1}t_2^{-1}\mathcal{H}om^{\operatorname{fil}}(\mathcal{V}_\bullet^{\hat{1}},\mathcal{V}_\bullet^{\hat{1}})\twoheadrightarrow t_1^{-1}t_2^{-1}\mathcal{H}om^{\operatorname{fil}}(\mathcal{V}_\bullet^{\hat{1}},\overline{\mathcal{V}}_\bullet^{\hat{1}}).\]
Both maps are surjective, so that the transversality requirement is automatic. 
\end{proof}
Next, we focus on the map forgetting the largest subscheme. Here, we are able to show a bit more and provide a factorization of the virtual structure sheaf.
\begin{lemma} \label{factorization-existence}
Let $\underline{d}=(d_0,\dots ,d_{r})\in \Z_{\geq 0}^{r+1}$ be a dimension vector. Set
\[\underline{\hat{d}}=(d_0,\dots ,d_{r-1})\in \Z_{\geq 0}^r\]
and let  
\[p=p_{\underline{d},\underline{\hat{d}}}:\operatorname{ncNHilb}^{\underline{d}}(\A^2)\to \operatorname{ncNHilb}^{\underline{\hat{d}}}(\A^2)\]
be the map given by the non-commutative tuple $(\mathcal{V}_\bullet/\mathcal{V}_r,\Phi_1,\Phi_2,\mathfrak{v})$. 
There exists a map of equivariant bundles 
\[\psi:\mathcal{E}_{\underline{d}}\to p^*\mathcal{E}_{\underline{\hat{d}}}\]
such that the pair $(p,\psi)$ satisfies the conditions of \cref{relative POT criterion}. In particular, the forgetful map 
\[p:\operatorname{NHilb}^{\underline{d}}(\A^2)\to \operatorname{NHilb}^{\underline{\hat{d}}}(\A^2)\]
can be equipped with a relative perfect obstruction theory compatible with the obstruction theories defined in \cref{POT DEF}. Furthermore, there exists a class $\mathcal{E}_{\underline{d}/\underline{\hat{d}}}\in K_T^0(\operatorname{NHilb}^{\underline{d}}(\A^2))$ such that 
\[\oo^{\operatorname{vir}}_{\underline{d}}=\mathcal{E}_{\underline{d}/\underline{\hat{d}}}\cdot p^*\mathcal{O}_{\underline{\hat{d}}}^{\operatorname{vir}}.\]
In particular, by the projection formula we get that
\[p_*\oo^{\operatorname{vir}}_{\underline{d}}=p_*(\mathcal{E}_{\underline{d}/\underline{\hat{d}}})\cdot \mathcal{O}_{\underline{\hat{d}}}^{\operatorname{vir}}.\]
\end{lemma}
\begin{proof}
We simply define $\psi$ as the quotient map
\[\psi:\mathcal{E}_{\underline{d}}=t_1^{-1}t_2^{-1}\mathcal{H}om^{\operatorname{fil}}(\mathcal
{V}_\bullet,\overline{\mathcal{V}}_{\bullet})\to t_1^{-1}t_2^{-1}\mathcal{H}om^{\operatorname{fil}}(\mathcal
{V}_\bullet/\mathcal{V}_r,\overline{\mathcal{V}}_{\bullet}/\overline{\mathcal{V}}_r)=p^*\mathcal{E}_{\underline{\hat{d}}}\]
which clearly takes the commutator of the universal operators to itself. Surjectivity of $\psi$ automatically implies the transversality requirement, so that the pair $(p,\psi)$ satisfies the conditions of \cref{relative POT criterion}. If we let $\mathcal{F}=\operatorname{ker}(\psi)$ we have 
\[\mathcal{E}_{\underline{d}}=\mathcal{F}+p^*\mathcal{E}_{\underline{\hat{d}}}\in K_T^0(\operatorname{ncNHilb}^{\underline{d}}(\A^2))\]
and so
\[\wedge_{-1}^{\bullet}(\mathcal{E}^{\vee}_{\underline{d}})=\wedge_{-1}^{\bullet}(\mathcal{F}^{\vee})\cdot p^*(\wedge_{-1}^{\bullet}(\mathcal{E}^{\vee}_{\underline{\hat{d}}}))\]
hence $\mathcal{E}_{\underline{d}/\underline{\hat{d}}}:=\wedge_{-1}^{\bullet}(\mathcal{F}^{\vee})|_{\operatorname{NHilb}^{\underline{d}}(\A^2)}$ has the desired property.
\end{proof}

\begin{remark}
In the case where we forget the smallest subscheme we get a similar but slightly different factorization of the virtual structure sheaf. Let $f,\psi$ be as in \cref{forget small subscheme}. When $d_0=d_1=0$ we see that $\psi$ is an isomorphism, so that we have
\[\mathcal{O}_{\underline{d}}^{\operatorname{vir}}=f^*\oo_{\underline{d}^{(1)}}^{\operatorname{vir}}.\]
When $d_0=0$ and $d_1\neq0$ we see that $\psi$ is a surjection so by setting $\mathcal{E}_{\underline{d}/\underline{d}^{(1)}}=\wedge^{\bullet}_{-1}(\operatorname{ker}\psi)^{\vee}$ we get
\[\mathcal{O}_{\underline{d}}^{\operatorname{vir}}=\mathcal{E}_{\underline{d}/\underline{d}^{(1)}}\cdot f^*\oo_{\underline{d}^{(1)}}^{\operatorname{vir}}.\]
In the case where $d_0\neq 0$ we see that the map $\psi$ is injective so by setting $\mathcal{K}_{\underline{d}/\underline{d}^{(1)}}=\wedge^{\bullet}_{-1}(\operatorname{coker}\psi)^{\vee}$ we get
\[\mathcal{K}_{\underline{d}/\underline{d}^{(1)}}\cdot \mathcal{O}_{\underline{d}}^{\operatorname{vir}}=f^*\oo_{\underline{d}^{(1)}}^{\operatorname{vir}}.\]
\end{remark}
\section{Localization calculations}
\subsection{Localization of the factorization}In this section we directly compute $p_*(\mathcal{E}_{\underline{d}/\underline{\hat{d}}})$ via localization. Note that if we can prove that
\[p_*(\mathcal{E}_{\underline{d}/\underline{\hat{d}}})=q^*\chi(\operatorname{Hilb}^{d_r}(\A^2),\wedge_{-1}^{\bullet}(\omega_{\A^2}^{[d_r]}))\]
then this will immediately imply \cref{Main theorem}. Recall that we proved in \cref{GSY comparison} that 
\[T^{\operatorname{vir}}_{\underline{d}}=\mathbb{E}_{\underline{d}}^{\vee}=\mathcal{Q}_{r+1}+\frac{\mathcal{Q}_1^\vee}{t_1t_2}-\frac{(1-t_1)(1-t_2)}{t_1t_2}\sum_{j=0}^{r}\mathcal{Q}_{j+1}^{\vee}(\mathcal{Q}_{j+1}-\mathcal{Q}_{j})\]
where $\mathcal{Q}_i$ is the tautological bundle of the $i$'th subscheme. In particular, if we for a nested Young diagram $\lambda_\bullet$ set
\[Q_{\lambda_i}=\mathcal{O}_{\C[x,y]/I_{\lambda_i}}\in K_T^0(\operatorname{pt})\]
we get that
\[T_{\underline{d},\lambda_\bullet}^{\operatorname{vir},\vee}=\operatorname{Q}_{\lambda_{r+1}}^{\vee}+t_1t_2Q_{\lambda_1}-(1-t_1)(1-t_2)\sum_{j=0}^rQ_{\lambda_{j+1}}(Q_{\lambda_{j+1}}^{\vee}-Q_{\lambda_{j}}^{\vee}).\]
\begin{lemma} \label{correction}
    Let $\underline{d}=(d_0,\dots ,d_{r})\in \Z_{\geq 0}^r$ be a dimension vector and let $\mathcal{E}_{\underline{d}/\underline{\hat{d}}}\in K_T^0(\operatorname{NHilb}^{\underline{d}}(\A^2))$ be the class of \cref{factorization-existence}. Let $\lambda_{\bullet}$ be a nested partition of type $\underline{d}$ and let 
    \[\operatorname{cont}_{\lambda_\bullet}(\mathcal{E}_{\underline{d}/\underline{\hat{d}}})\in K_T^0(\operatorname{pt})\]
    be the contribution of the Thomason localization of $\mathcal{E}_{\underline{d}/\underline{\hat{d}}}$ at the fixed point corresponding to $\lambda_\bullet$. Then
    \[\operatorname{cont}_{\lambda_\bullet}(\mathcal{E}_{\underline{d}/\underline{\hat{d}}})=\frac{1}{\wedge_{-1}^{\bullet}(\mathcal{C}_{\lambda_{r},\lambda _{r+1}})}\]
    where 
    \[\mathcal{C}_{\lambda_{r},\lambda _{r+1}}=(Q_{\lambda_{r+1}}^{\vee}-Q_{\lambda_{r}}^{\vee})(1-(1-t_1)(1-t_2)Q_{\lambda_{r+1}}).\]
    In particular, the contribution depends only on $\lambda_{r}\subset\lambda_{r+1}$.
\end{lemma}
\begin{proof}
By definition we see that 
\[\operatorname{cont}_{\lambda_\bullet}(\mathcal{E}_{\underline{d}/\underline{\hat{d}}})=\frac{1}{\wedge^{\bullet}_{-1}(T_{\underline{d},\lambda_\bullet}^{\operatorname{vir},\vee}-p^*T_{\underline{\hat{d}},\lambda_\bullet}^{\operatorname{vir},\vee})}.\]
We then directly see that
\begin{align*}T_{\underline{d},\lambda_\bullet}^{\operatorname{vir},\vee}&=\operatorname{Q}_{\lambda_{r+1}}^{\vee}+t_1t_2Q_{\lambda_1}-(1-t_1)(1-t_2)\sum_{j=0}^rQ_{\lambda_{j+1}}(Q_{\lambda_{j+1}}^{\vee}-Q_{\lambda_{j}}^{\vee})\\
&=\sum_{i=0}^r(Q_{\lambda_{j+1}}^{\vee}-Q_{\lambda_{j}}^{\vee})+t_1t_2Q_{\lambda_{1}}-(1-t_1)(1-t_2)\sum_{j=0}^rQ_{\lambda_{j+1}}(Q_{\lambda_{j+1}}^{\vee}-Q_{\lambda_{j}}^{\vee})\\
&=Q_{\lambda_{1}}^{\vee}+t_1t_2Q_{\lambda_{1}}-(1-t_1)(1-t_2)Q_{\lambda_{1}}^{\vee}Q_{\lambda_{1}}+\sum_{i=1}^r(Q_{\lambda_{j+1}}^{\vee}-Q_{\lambda_{j-1}}^{\vee})(1-(1-t_1)(1-t_2)Q_{\lambda_{j+1}}).
\end{align*}
In particular, we see that
\[T_{\underline{d},\lambda_\bullet}^{\operatorname{vir},\vee}-p^*T_{\underline{\hat{d}},\lambda_\bullet}^{\operatorname{vir},\vee}=(Q_{\lambda_{r+1}}^{\vee}-Q_{\lambda_{r}}^{\vee})(1-(1-t_1)(1-t_2)Q_{\lambda_{r+1}})\]
as desired.
\end{proof}
With this in mind, we make the following definition.
\begin{definition}
    For any nesting of Young diagrams $\lambda\subset \mu$, we define the expression
    \[P_{\mu/\lambda}=\frac{1}{\wedge^{\bullet}_{-1}(\mathcal{C}_{\lambda,\mu})}\]
    and for fixed $\lambda$ we let 
    \[F_{\lambda}(u)=\sum_{\lambda\subset \mu}u^{|\mu|-|\lambda|}P_{\mu/\lambda}\in \Q(t_1,t_2)[[u]].\]
\end{definition}
\begin{remark}
        We note that for a nested partition $\lambda_{\bullet}=\lambda_1\subset \cdots \subset \lambda_{r}$ of type $\underline{\hat{d}}$ and
        \[p:\operatorname{NHilb}^{\underline{d}}(\A^2)\to \operatorname{NHilb}^{\underline{\hat{d}}}(\A^2)\]
        we get
    \[\operatorname{cont}_{\lambda_{\bullet}}(p_*(\mathcal{E}_{\underline{d}/\underline{\hat{d}}}))=\sum_{\underset{|\mu|-|\lambda_{r}|=d_r}{\lambda_{r}\subset \mu}}P_{\mu/\lambda_{r}}=[u^{d_r}]F_{\lambda_{r}}(u).\]
    We also note that for a Young diagram $\mu$ with $|\mu|=d$ we have
    \[P_{\mu/\varnothing}=\operatorname{cont}_{\mu}(\wedge_{-1}^{\bullet}\omega_{\A^2}^{[d]})\]
    so that
    \[F_{\varnothing}(u)=\sum_{d\geq 0}u^d\chi_{T}(\operatorname{NHilb}^{d}(\A^2),\wedge_{-1}^\bullet \left(\omega_{\A^2}^{[d]}\right))=\operatorname{PE}\left(\frac{u(1-t_1t_2)}{(1-t_1)(1-t_2)}\right).\]
    We conclude that \cref{Main theorem} follows directly if we can prove that
    \[F_{\lambda}(u)=F_{\varnothing}(u).\]
\end{remark}
\subsection{Independence of the generating series}
We will now prove that $F_{\lambda}(u)$ is independent of $\lambda$ finishing the proof \cref{Main theorem}. We start by introducing some notation to ease the later calculations.
\begin{set}
    For any subset $A\subset \Z^2$ we write $x,y,z, v,w,\dots\in A$ when we really mean 
    \[x,y,z,v,w,\dots=t_1^{a_1}t_2^{a_2}\in \Z[t_1^{\pm},t_2^{\pm}]\] 
    for some $(a_1,a_2)\in A$. In this notation it follows that we have
    \[\mathcal{C}_{\lambda,\mu}=\sum_{z\in \mu\backslash\lambda}z^{-1}\left(1-(1-t_1)(1-t_2)\sum_{v\in \mu}v\right)\]
    for a nesting $\lambda\subset \mu$ of Young diagrams. For a Young diagram $\lambda$ and boxes $x\in \lambda$ and $y\in \Z^2\backslash\lambda$ we will say that $x$ is a removable box if $\lambda \backslash \{x\}$ is still a Young diagram and we will say that $y$ is an addable box if $\lambda \cup \{y\}$ is still a partition. We denote by $R(\lambda)$, $A(\lambda)$ the set of removable respectively addable boxes of $\lambda$. Note that
    \[|R(\lambda)|+1=|A(\lambda)|\]
    as can easily be proven by induction on the number of boxes of $|\lambda|$.
    We will typically abuse notation by identifying the elements of these sets by their algebraic elements above. For example, for the Young diagram $\lambda$ of the form
    
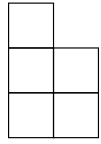
\begin{center}
     \begin{ytableau}
\none & & \none \\
\none &  &  \\
\none & & \\
\end{ytableau}
    \end{center}
we would write 
\begin{align*}
    R(\lambda)=\{t_1t_2,t_2^2\}, &&A(\lambda)=\{t_1^2,t_1t_2^2,t_2^3\}.
\end{align*}
We will further define $\mathfrak{t}=t_1t_2$ and 
\[C=\frac{(1-\mathfrak{t})}{(1-t_1)(1-t_2)}=\chi_T(\A^2,\wedge_{-1}^{\bullet}(\omega_{\A^2}))\]
and 
\[\zeta(x)=\frac{(1-x)(1-\mathfrak{t}x)}{(1-t_1x)(1-t_2x)}\]
since these expressions will appear a lot in our calculations below.
\end{set}
Note that for a nesting of Young diagrams $\lambda \subset \mu$  we directly see that
\[P_{\mu/\lambda}=\frac{1}{\wedge^{\bullet}_{-1}(\mathcal{C}_{\lambda,\mu})}=\prod_{z\in \mu\backslash\lambda}\frac{1}{(1-z^{-1})}\prod_{v\in \mu}\frac{(1-\frac{v}{z})(1-\mathfrak{t}\frac{v}{z})}{(1-t_1\frac{v}{z})(1-t_2\frac{v}{z})}\]
that is 
\[\prod_{z\in \mu\backslash\lambda }\frac{1}{(1-z^{-1})}\prod_{v\in \mu}\zeta\left(\frac{v}{z}\right).\]
There is a lot of cancellation happening in this expression as the following lemma shows.

\begin{lemma}
For any Young diagram $\mu$ we have 
    \[\frac{1}{(1-z^{-1})}\prod_{v\in \mu}\frac{(1-\frac{v}{z})(1-t_1t_2\frac{v}{z})}{(1-t_1\frac{v}{z})(1-t_2\frac{v}{z})}=\frac{ \prod_{x\in R(\mu)}(1-t_1t_2\frac{x}{z})}{ \prod_{y\in A(\mu)}(1-\frac{y}{z})}\]
where the identity is of meromorphic functions in the variable $z$.
\end{lemma}
\begin{proof}
    We proceed by induction on $|\mu|$. For $\mu=\varnothing$ the left hand side is 
    \[\frac{1}{(1-z^{-1})}.\]
    We have $R(\varnothing)=\varnothing$ while $A(\varnothing)=\{1\}$ so that the right hand side is
    \[\frac{1}{(1-\frac{1}{z})}\]
    as desired. Suppose the statement is true for $\mu$ and let $w\in A(\mu)$ with $\mu'=\mu\cup \{w\}$. The induction hypothesis now immediately implies that
     \[\frac{1}{(1-z^{-1})}\prod_{v\in \mu'}\frac{(1-\frac{v}{z})(1-t_1t_2\frac{v}{z})}{(1-t_1\frac{v}{z})(1-t_2\frac{v}{z})}=\frac{(1-\frac{w}{z})(1-t_1t_2\frac{w}{z})}{(1-t_1\frac{w}{z})(1-t_2\frac{w}{z})}\frac{ \prod_{x\in R(\mu)}(1-t_1t_2\frac{x}{z})}{ \prod_{y\in A(\mu)}(1-\frac{y}{z})}.\]
    We consider 4 different cases for the placement of the box $w$
    
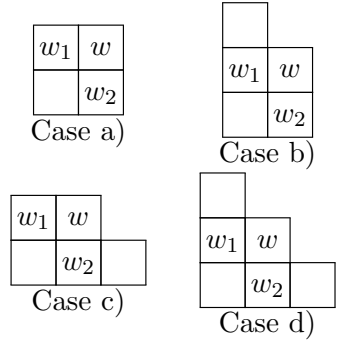
    \begin{center}
            \begin{tabular}{cc}
        \begin{tabular}{c}
     \begin{ytableau}
         w_1 & w \\
         \, &  w_2
     \end{ytableau}
      \\ 
      Case a)
\end{tabular} & \begin{tabular}{c}
     \begin{ytableau}
        \, & \none \\
         w_1 & w \\
         \, &  w_2
     \end{ytableau}
      \\ 
      Case b)
\end{tabular} \\
        \begin{tabular}{c}
     \begin{ytableau}

         w_1 & w \\
         \, & w_2 & 
     \end{ytableau}
      \\ 
      Case c)
\end{tabular} & \begin{tabular}{c}
     \begin{ytableau}
       \, & \none \\
        w_1 & w \\
         \, & w_2 & 
     \end{ytableau}
      \\ 
      Case d)
\end{tabular}
    \end{tabular}
    \end{center}
where we denote the neighbors of $w$ by $w_1$ and $w_2$. Note that $t_1w_1=w$ and $t_2w_2=w$. Note that case b), d) respectively case c), d) also includes the case where $w$ is on the $y$- respectively $x$-axis.
\begin{enumerate}[label=\alph*)]
    \item Here we have
    \begin{align*}
        R(\mu')=(R(\mu)\backslash\{w_1,w_2\})\cup \{w\}, && A(\mu')=A(\mu)\backslash \{w\}
    \end{align*}
    so that
    \begin{align*}
        \frac{ \prod_{x\in R(\mu)}(1-t_1t_2\frac{x}{z})}{ \prod_{y\in A(\mu)}(1-\frac{y}{z})}&=\frac{(1-t_1t_2\frac{w_1}{z})(1-t_1t_2\frac{w_2}{z})(1-t_1t_2\frac{w}{z})^{-1}  \prod_{x\in R(\mu')}(1-t_1t_2\frac{x}{z})}{ (1-\frac{w}{z})\prod_{y\in A(\mu')}(1-\frac{y}{z})}\\
        &=\frac{(1-t_2\frac{w}{z})(1-t_1\frac{w}{z})  \prod_{x\in R(\mu')}(1-t_1t_2\frac{x}{z})}{ (1-\frac{w}{z})(1-t_1t_2\frac{w}{z})\prod_{y\in A(\mu')}(1-\frac{y}{z})}
    \end{align*}
    which is what we want.
    \item Here we have
    \begin{align*}
        R(\mu')=(R(\mu)\backslash\{w_2\})\cup \{w\}, && A(\mu')=(A(\mu)\backslash \{w\})\cup\{t_2w\}
    \end{align*}
    so that
    \begin{align*}
        \frac{ \prod_{x\in R(\mu)}(1-t_1t_2\frac{x}{z})}{ \prod_{y\in A(\mu)}(1-\frac{y}{z})}&=\frac{(1-t_1t_2\frac{w_2}{z})(1-t_1t_2\frac{w}{z})^{-1}  \prod_{x\in R(\mu')}(1-t_1t_2\frac{x}{z})}{ (1-\frac{w}{z})(1-t_2\frac{w}{z})^{-1}\prod_{y\in A(\mu')}(1-\frac{y}{z})}\\
        &=\frac{(1-t_2\frac{w}{z})(1-t_1\frac{w}{z})  \prod_{x\in R(\mu')}(1-t_1t_2\frac{x}{z})}{ (1-\frac{w}{z})(1-t_1t_2\frac{w}{z})\prod_{y\in A(\mu')}(1-\frac{y}{z})}
    \end{align*}
    which is what we want.
    \item Here we have
    \begin{align*}
        R(\mu')=(R(\mu)\backslash\{w_1\})\cup \{w\}, && A(\mu')=(A(\mu)\backslash \{w\})\cup\{t_1w\}
    \end{align*}
    so that
    \begin{align*}
        \frac{ \prod_{x\in R(\mu)}(1-t_1t_2\frac{x}{z})}{ \prod_{y\in A(\mu)}(1-\frac{y}{z})}&=\frac{(1-t_1t_2\frac{w_1}{z})(1-t_1t_2\frac{w}{z})^{-1}  \prod_{x\in R(\mu')}(1-t_1t_2\frac{x}{z})}{ (1-\frac{w}{z})(1-t_1\frac{w}{z})^{-1}\prod_{y\in A(\mu')}(1-\frac{y}{z})}\\
        &=\frac{(1-t_2\frac{w}{z})(1-t_1\frac{w}{z})  \prod_{x\in R(\mu')}(1-t_1t_2\frac{x}{z})}{ (1-\frac{w}{z})(1-t_1t_2\frac{w}{z})\prod_{y\in A(\mu')}(1-\frac{y}{z})}
    \end{align*}
    which is what we want.
    \item Here we have
    \begin{align*}
R(\mu')=R(\mu)\cup\{w\}, && A(\mu')=(A(\mu)\backslash \{w\})\cup\{t_1w,t_2w\}
    \end{align*}
    so that
    \begin{align*}
        \frac{ \prod_{x\in R(\mu)}(1-t_1t_2\frac{x}{z})}{ \prod_{y\in A(\mu)}(1-\frac{y}{z})}&=\frac{(1-t_1t_2\frac{w}{z})^{-1}  \prod_{x\in R(\mu')}(1-t_1t_2\frac{x}{z})}{ (1-\frac{w}{z})(1-t_1\frac{w}{z})^{-1}(1-t_2\frac{w}{z})^{-1}\prod_{y\in A(\mu')}(1-\frac{y}{z})}\\
        &=\frac{(1-t_2\frac{w}{z})(1-t_1\frac{w}{z})  \prod_{x\in R(\mu')}(1-t_1t_2\frac{x}{z})}{ (1-\frac{w}{z})(1-t_1t_2\frac{w}{z})\prod_{y\in A(\mu')}(1-\frac{y}{z})}
    \end{align*}
    which is what we want.
\end{enumerate}
\end{proof}

\begin{definition}
With this in mind we define
\[\omega_{\mu}(x)=\frac{1}{1-x^{-1}}\prod_{z\in \mu}\zeta\left(\frac{z}{x}\right)=\frac{\prod_{r\in R(\mu)}(1-\mathfrak{t}\frac{r}{x})}{\prod_{a\in A(\mu)}(1-\frac{a}{x})}\]
as a meromorphic function. For a nesting of partitions $\lambda\subset \mu$ we therefore have
\[P_{\mu/\lambda}=\prod_{z\in \mu\backslash \lambda}\omega_\mu(z).\]
\end{definition}

\begin{rec}
We note that $\omega$ satisfies the branching rule, where for $w\in A(\mu)$ we have 
\begin{align*}\omega_{\mu\cup \{w\}}(x)&=\zeta\left(\frac{w}{x}\right)\omega_{\mu}(x)\\
&=\frac{(1-\frac{w}{x})(1-\mathfrak{t}\frac{w}{x})}{(1-t_1\frac{w}{x})(1-t_2\frac{w}{x})}\frac{\prod_{r\in R(\mu)}(1-\mathfrak{t}\frac{r}{x})}{\prod_{a\in A(\mu)}(1-\frac{a}{x})}\\
&=\frac{(1-\mathfrak{t}\frac{w}{x})}{(1-t_1\frac{w}{x})(1-t_2\frac{w}{x})}\frac{\prod_{r\in R(\mu)}(1-\mathfrak{t}\frac{r}{x})}{\prod_{a\in A(\mu)\backslash \{w\}}(1-\frac{a}{x})}
\end{align*}
and so in particular
\[\omega_{\mu\cup \{w\}}(w)=C\frac{\prod_{r\in R(\mu)}(1-\mathfrak{t}\frac{r}{w})}{\prod_{a\in A(\mu)\backslash \{w\}}(1-\frac{a}{w})}.\]
We note also the following symmetry of $\zeta$
\[\zeta\left(\frac{1}{\mathfrak{t}x}\right)=\frac{(1-\frac{1}{\mathfrak{t}x})(1-\mathfrak{t}\frac{1}{\mathfrak{t}x})}{(1-t_1\frac{1}{\mathfrak{t}x})(1-t_2\frac{1}{\mathfrak{t}x})}=\frac{\mathfrak{t}^{-1}x^{-2}(\mathfrak{t}x-1)(x-1)}{t_2^{-1}t_1^{-1}x^{-2}(t_2x-1)(t_1x-1)}=\zeta(x).\]
We will frequently use these rules in the computations below.
\end{rec}

Recall that our goal is to show that $F_{\lambda}(u)$ is independent of $\lambda$, so that 
\[F_{\lambda}(u)=F_{\varnothing}(u)=\operatorname{PE}\left(\frac{u(1-t_1t_2)}{(1-t_1)(1-t_2)}\right).\]
Suppose therefore $w\in A(\lambda)$ is an addable box and let $\lambda'=\lambda \cup \{w\}$ be the diagram obtained by adding $w$ to $\lambda$. We want to prove that $F_{\lambda}(u)=F_{\lambda'}(u)$. We can split the sum as 
\[F_{\lambda}(u)=\sum_{\lambda'\subset \mu} u^{|\mu|-|\lambda|}P_{\mu/\lambda}+\sum_{\underset{w\notin \nu}{\lambda \subset \nu}}u^{|\nu|-|\lambda|}P_{\nu/\lambda}.\]
Sending $\nu\mapsto \nu\cup \{w\}$ provides a bijection
\[\{\lambda \subset \nu \:| \: w\notin \nu\}\cong \{\lambda \subset \mu \: | \: w\in R(\omega)\}.\]
We note further that
\[P_{\mu/\lambda}=\prod_{z\in \mu\backslash\lambda}\omega_{\mu}(z)=\omega_{\mu}(w)\prod_{z\in \mu\backslash\lambda'}\omega_{\mu}(z)=\omega_{\mu}(w)P_{\mu/\lambda'}\]
and for $\mu=\nu\cup \{w\}$ we have
\[P_{\mu/\lambda'}=\prod_{z\in (\nu\cup\{w\})\backslash(\lambda\cup\{w\})}\omega_{\nu\cup\{w\}}(z)=\prod_{z\in\nu \backslash \lambda }\zeta(w/z)\omega_{\nu}(z)=P_{\nu/\lambda}\prod_{z\in\nu \backslash \lambda }\zeta(w/z)\]
so setting 
\[\kappa_{\mu/\lambda'}(x)=\prod_{z\in \mu\backslash \lambda'}\zeta(x/z)^{-1}=\prod_{z\in \mu\backslash \lambda'}\frac{(1-t_1x/z)(1-t_2x/z)}{(1-x/z)(1-\mathfrak{t}x/z)}\]
we obtain 
\[P_{\nu/\lambda}=\kappa_{\mu/\lambda'}(w)P_{\mu/\lambda'}.\]
Using this, it follows that we may rewrite 
\[F_{\lambda}(u)=\sum_{\lambda'\subset \mu}(u\cdot \omega_{\mu}(w)) u^{|\mu|-|\lambda'|}P_{\mu/\lambda'}+\sum_{\underset{w\in R(\mu)}{\lambda' \subset \mu}}\kappa_{\mu/\lambda'}(w)u^{|\mu|-|\lambda'|}P_{\mu/\lambda'}.\]
We note that if $w\notin R(\mu)$ for $\lambda'\subset \mu$ we have $\kappa_{\mu/\lambda'}(w)=0$. Indeed, since $w\notin R(\mu)$ we must have $t_1w\in \mu$ or/and $t_2w\in\mu$ so that the numerator of $\kappa_{\mu/\lambda'}(x)$ has a zero of degree 1 or 2 at $x=w$. Conversely, since $w\in \lambda'$ we see that the denominator has a zero - which is necessarily of degree 1 - at $x=w$ if and only if $t_1t_2w\in \mu$, but then $t_1w,t_2w\in \mu$ so that the numerator has a zero of degree 2. It follows that
\[F_{\lambda}(u)=\sum_{\lambda'\subset \mu}(u\cdot \omega_{\mu}(w)+\kappa_{\mu/\lambda'}(w)) u^{|\mu|-|\lambda'|}P_{\mu/\lambda'}\]
and in particular, we see that
\[F_{\lambda}(u)-F_{\lambda'}(u)=\sum_{\lambda'\subset \mu}(u\cdot \omega_{\mu}(w)+\kappa_{\mu/\lambda'}(w)-1) u^{|\mu|-|\lambda'|}P_{\mu/\lambda'}.\]
This motivates the following definition.
\begin{definition}
Let $\lambda$ be a Young diagram. We set
\[D_{\lambda}(x)=\sum_{\lambda\subset \mu}(u\cdot \omega_{\mu}(x)+\kappa_{\mu/\lambda}(x)-1) u^{|\mu|-|\lambda|}P_{\mu/\lambda}\]
and
\[D_{\lambda,n}(x)=\sum_{\underset{|\mu\backslash\lambda|=n}{\lambda \subset \mu}}P_{\mu/\lambda}(\kappa_{\mu/\lambda}(x)-1)+\sum_{\underset{|\nu\backslash\lambda|=n-1}{\lambda \subset \nu}}P_{\nu/\lambda}\omega_{\nu}(x)\]
the coefficient of $u^n$ in $D_{\lambda}(u)$. Similarly we let 
\[c_n(\lambda)=\sum_{\underset{|\mu|-|\lambda|=n}{\lambda\subset \mu}}P_{\mu/\lambda}\]
the coefficient of $u^n$ in $F_{\lambda}(u)$.
\end{definition}
The above discussion now shows the following lemma.
\begin{lemma}
    Let $\lambda $ be a partition and let $w\in R(\lambda)$ be a removable box. Then $c_n(\lambda)=c_n(\lambda\backslash\{w\})$ if and only if $D_{\lambda,n}(w)=0$. In particular, if $D_{\lambda,n}(w)=0$ for all $\lambda$ and all $w\in R(\lambda)$ then $c_n=c_n(\lambda)$ is independent of $\lambda$. 
\end{lemma}
With this in mind we begin the study of $D_{\lambda,n}(x)$. We start by identifying its poles.

\begin{lemma}
    $D_{\lambda,n}(x)$ is a meromorphic function. It has poles exactly at $x=a/\mathfrak{t}$ for $a\in A(\lambda)$ and these poles are simple with residue 
    \[\operatorname{Res}_{x=a/\mathfrak{t}}D_{\lambda,n}(x)=\frac{a\cdot \omega_{\lambda\cup\{a\}}(a)}{\mathfrak{t}C}\cdot c_{n-1}(\lambda\cup\{a\}).\]
\end{lemma}
\begin{proof}
    Note first that 
\[\zeta(x/z)=\zeta(z/\mathfrak{t}x)\]
and so 
\[\omega_{\lambda}(\mathfrak{t}x)\cdot \kappa_{\mu/\lambda}(x)^{-1}=\frac{1}{1-(\mathfrak{t}x)^{-1}}\prod_{z\in \lambda}\zeta(z/\mathfrak{t}x)\cdot \prod_{z\in \mu\backslash\lambda}\zeta(z/\mathfrak{t}x)=\omega_{\mu}(\mathfrak{t}x)\]
so that
\[\kappa_{\mu/\lambda}(x)=\frac{\omega_\lambda(\mathfrak{t}x)}{\omega_{\mu}(\mathfrak{t}x)}=\frac{\prod_{r\in R(\lambda)}(1-r/x)\prod_{a\in A(\mu)}(1-a/\mathfrak{t}x)}{\prod_{a\in A(\lambda)}(1-a/\mathfrak{t}x)\prod_{r\in R(\mu)}(1-r/x)}\]
i.e. the possible poles are for $x=b\in R(\mu)\backslash R(\lambda )$ or $x=a/\mathfrak{t}$ for $a\in A(\lambda)$.  Note that if $a\in A(\lambda)$ then $a/t_1,a/t_2\in \lambda$ hence $a/\mathfrak{t}\in \lambda\backslash R(\lambda)$. In particular $a/\mathfrak{t}\notin R(\mu)$ so all these possible poles are distinct and at most of order 1. On the other hand we have 
\[\omega_{\nu}(x)=\frac{\prod_{r\in R(\nu)}(1-\mathfrak{t}r/x)}{\prod_{a\in A(\nu)}(1-a/x)}\]
which has simple poles at $x=b\in A(\nu)$. Note that if we choose any box $b$ outside $\lambda$ which is addable for a $\nu$ or removable for a $\mu$ we have a bijection 
\[\{\nu\supset \lambda\:| \: |\nu\backslash\lambda|=n-1,\, b\in A(\nu)\}\longleftrightarrow \{\mu\supset \lambda\:| \: |\nu\backslash\lambda|=n,\, b\in R(\mu)\backslash R(\lambda)\}\]
simply given by either removing or adding the box $b$. For a fixed $b$ we will show that the residue cancels for such a pair $(\mu,\nu)$ with $\mu=\nu\cup \{b\}$. We note that
\[\operatorname{Res}_{x=b}\omega_{\nu}(x)dx=b\frac{\prod_{r\in R(\nu)}(1-\mathfrak{t}r/b)}{\prod_{a\in A(\nu)\backslash\{b\}}(1-a/b)}=\frac{b}{C}\omega_{\mu}(b)\]
where $\mu=\nu\cup \{b\}$ and where we have used the branching rule. On the other hand
\[\operatorname{Res}_{x=b}\zeta(x/b)^{-1}dx=-\frac{b}{C}\]
so that 
\[\operatorname{Res}_{x=b}\kappa_{\mu/\lambda}(x)=-\frac{b}{C}\prod_{z\in \nu\backslash\lambda }\zeta(b/z)^{-1}.\]
Hence the two summands of $\operatorname{Res}_{x=b}D_{\lambda,n}(x)$ corresponding to $\mu$ and $\nu$ with $\mu=\nu\cup \{b\}$ are given by
\[P_{\mu/\lambda}\left(-\frac{b}{C}\prod_{z\in \nu\backslash\lambda }\zeta(b/z)^{-1}\right)+P_{\nu/\lambda}\frac{b}{C}\omega_{\mu}(b)\]
but
\[P_{\mu/\lambda}=\prod_{z\in \nu\cup \{b\}/\lambda}\omega_{\nu\cup \{b\}}(z)=\omega_{\mu}(b)\prod_{z\in \nu}\zeta(b/z)\omega_{\nu}(z)=\omega_{\mu}P_{\nu/\lambda}(b)\prod_{z\in \nu}\zeta(b/z)\]
which shows that $\operatorname{Res}_{x=b}D_{\lambda,n}(x)=0$. For the remaining possible poles we see that for $a\in A(\lambda)$ we get a pole for each $\mu$ such that $a\notin A(\mu)$ i.e. $a\in \mu$. We see that
\[\operatorname{Res}_{x=a/\mathfrak{t}}\zeta(x/a)^{-1}dx=\frac{a}{\mathfrak{t}C}\]
and so for $\mu$ with $\lambda \cup \{a\}\subset \mu$ we get 
\[\operatorname{Res}_{x=a/\mathfrak{t}}\kappa_{\mu/\lambda}(x) dx=\frac{a}{\mathfrak{t}C}\prod_{z\in \mu \backslash (\lambda\cup \{a\})}\zeta(a/\mathfrak{t}z)^{-1}=\frac{a}{\mathfrak{t}C}\prod_{z\in \mu \backslash (\lambda\cup \{a\})}\zeta(z/a)^{-1}.\]
We note that 
\[P_{\mu/\lambda}=\prod_{z\in \mu\backslash \lambda}\omega_{\mu}(z)=\omega_{\mu}(a)P_{\mu/(\lambda\cup \{a\})}\]
and that
\[w_{\mu}(a)=\prod_{z\in \mu}\zeta(z/a)\]
so that
\[\omega_{\mu}(a)\prod_{z\in \mu \backslash (\lambda\cup \{a\})}\zeta(z/a)^{-1}=\omega_{\lambda \cup a}(a)\]
hence we get 
\[\operatorname{Res}_{x=a/\mathfrak{t}}D_{\lambda,n}(x)=\frac{a\omega_{\lambda\cup\{a\}}(a)}{\mathfrak{t}C}\sum_{\underset{|\mu|-|\lambda \cup \{a\}|=n-1}{\lambda\cup \{a\}\subset \mu}}P_{\mu/(\lambda\cup \{a\})}=\frac{a\omega_{\lambda\cup\{a\}}(a)}{\mathfrak{t}C}\cdot c_{n-1}(\lambda\cup\{a\})\]
\end{proof}
Next, we identify the limit.
\begin{lemma}
    Let $\lambda$ be a diagram. We have 
    \[\lim_{x\to \infty} D_{\lambda,n}(x)=c_{n-1}(\lambda).\]
\end{lemma}
\begin{proof}
We see that
\[\kappa_{\mu/\lambda}(x)=\prod_{z\in \mu\backslash \lambda}\zeta(x/z)^{-1}=\prod_{z\in \mu\backslash \lambda}\frac{(1-t_1x/z)(1-t_2x/z)}{(1-x/z)(1-\mathfrak{t}x/z)}\]
which tends to 
\[\prod_{z\in \mu\backslash \lambda}\frac{(t_1/z)(t_2/z)}{(1/z)(\mathfrak{t}/z)}=1\]
when $x\to \infty$. Meanwhile
\[\omega_{\nu}(x)=\frac{\prod_{r\in R(\mu)}(1-\mathfrak{t}\frac{r}{x})}{\prod_{a\in A(\mu)}(1-\frac{a}{x})}\to 1\]
as $x\to \infty$. It follows that 
\[D_{\lambda,n}(x)\to\sum_{\underset{|\nu\backslash\lambda|=n-1}{\lambda \subset \nu}}P_{\nu/\lambda}=c_{n-1}(\lambda). \] 
\end{proof}
\begin{theorem}
    The expression $c_n(\lambda)$ is independent of $\lambda$ for all $n\geq 0$. \label{main combinatorial theorem}
\end{theorem}
\begin{proof}
    We proceed by induction on $n$ with the induction start $n=0$ being trivial as $c_0(\lambda)=1$. Assuming that $c_m=c_m(\lambda)$ is independent of $\lambda$ for $m<n$ we will show that
    \[D_{\lambda,n}(x)=xc_{n-1}\frac{\prod_{r\in R(\lambda)}(x-r)}{\prod_{a\in A(\lambda)}(x-a/\mathfrak{t})}\]
    which in particular implies that $D_{\lambda,n}(w)=0$ for all $w\in R(\lambda)$ which then implies that $c_n(\lambda)$ is independent of $\lambda$. By Liouville's theorem it will suffice to show that the right hand side has the same residues and limit at $x\to \infty$. Clearly the right hand side has simple poles at $x=b/\mathfrak{t}$ for $b\in A(\lambda)$ and the residue is given 
    \[c_{n-1}\frac{b}{\mathfrak{t}}\frac{\prod_{r\in R(\lambda)}(b/\mathfrak{t}-r)}{\prod_{a\in A(\lambda)\backslash\{b\}}(b/\mathfrak{t}-a/\mathfrak{t})}=c_{n-1}\frac{b}{\mathfrak{t}}\frac{\prod_{r\in R(\lambda)}(1-\mathfrak{t}r/b)}{\prod_{a\in A(\lambda)\backslash\{b\}}(1-a/b)}.\]
    Conversely we see using the branching rule that
    \[\omega_{\lambda\cup\{b\}}(b)=C\frac{\prod_{r\in R(\lambda)}(1-\mathfrak{t}r/b)}{\prod_{a\in A(\lambda)\backslash\{b\}}(1-a/b)}\]
    so that the residue of $D_{\lambda,n}(x)$ at $x=b/\mathfrak{t}$ is 
    \[\frac{b\cdot \omega_{\lambda\cup\{b\}}(b)}{\mathfrak{t}C}\cdot c_{n-1}(\lambda\cup\{b\})=\frac{b}{\mathfrak{t}}\frac{\prod_{r\in R(\lambda)}(1-\mathfrak{t}r/b)}{\prod_{a\in A(\lambda)\backslash\{b\}}(1-a/b)}c_{n-1}\]
    which agrees with the right hand side. Lastly, using $|A|=|R|+1$ we can rewrite the right hand side to
    \[c_{n-1}\frac{\prod_{r\in R(\lambda)}(1-r/x)}{\prod_{a\in A(\lambda)}(1-a/\mathfrak{t}x)} \]
    so that when $x\to \infty$ we get 
    \[c_{n-1}\]
    which agrees with 
    \[\lim_{x\to \infty}D_{\lambda,n}(x)=c_{n-1}(\lambda).\]
\end{proof}
\begin{proof}[Proof of \cref{Main theorem}]
Let $p=p_{\underline{d},\underline{\hat{d}}}$. We have by \cref{factorization-existence} that 
\[p_*(\mathcal{O}_{\underline{d}}^{\operatorname{vir}})=p_*(\mathcal{E}_{\underline{d}/\underline{\hat{d}}})\cdot \mathcal{O}_{\underline{\hat{d}}}^{\operatorname{vir}}.\]
For each nested Young diagram of type $\underline{d}$
\[\lambda_{\bullet}=\lambda_{r+1}\supset \lambda_{r-1}\supset \cdots \supset \lambda_1\]
corresponding to a fixed point of $\operatorname{NHilb}^{\underline{d}}(\A^2)$ we then see by \cref{correction} that
\[\operatorname{cont}_{\lambda_\bullet}(\mathcal{E}_{\underline{d}/\underline{\hat{d}}})=P_{\lambda_{r+1}/\lambda_{r}}\]
so that for a nested Young diagram $\lambda_{\bullet}$ of type $\underline{\hat{d}}$ we have
\[\operatorname{cont}_{\lambda_\bullet}p_*(\mathcal{E}_{\underline{d}/\underline{\hat{d}}})=\sum_{\underset{|\mu|-|\lambda_{r}|}{\lambda_{r}\subset \mu}}P_{\mu/\lambda_{r}}=c_{d_r}(\lambda_{r}).\]
By \cref{main combinatorial theorem} this is independent of $\lambda_{r}$, and so in particular $p_*(\mathcal{E}_{\underline{d}/\underline{\hat{d}}})$ is constant. One then explicitly sees that
\[c_{d_r}(\lambda_{r})=c_{d_r}(\varnothing)=\sum_{\mu \vdash d_r}P_{\mu/\varnothing}\]
agrees with the localization formula for $\chi_T(\operatorname{Hilb}^{d_r}(\A^2),\wedge_{-1}^\bullet(\omega_{\A^2}^{[d_r]}))$. Indeed, note first that 
\[\omega_{\A^2}^{[d_r]}=t_1t_2\oo_{\A^2}^{[d_r]}.\]
We have by definition that
\[P_{\mu/\varnothing}=\frac{1}{\wedge^{\bullet}_{-1}(\mathcal{C}_{\varnothing,\mu})}\]
where
\[\mathcal{C}_{\varnothing,\mu}=Q_{\mu}^{\vee}-(1-t_1)(1-t_2)Q_{\mu}^{\vee}Q_{\mu}.\]
Using that
\[T_{\operatorname{Hilb}^{d_r}(\A^2),\mu}^{\vee}=Q_{\mu}^{\vee}+t_1t_2Q_{\mu}-(1-t_1)(1-t_2)Q_{\mu}^{\vee}Q_{\mu}\]
we get
\[\mathcal{C}_{\varnothing,\mu}=T_{\operatorname{Hilb}^{d_r}(\A^2),\mu}^{\vee}-t_1t_2Q_{d_r}=T_{\operatorname{Hilb}^{d_r}(\A^2),\mu}^{\vee}-t_1t_2\mathcal{O}^{[d_r]}_{\A_2,\mu}\]
hence
\[P_{\mu/\varnothing}=\frac{\wedge_{-1}^{\bullet}(t_1t_2\mathcal{O}^{[d_r]}_{\A_2,\mu})}{\wedge^{\bullet}_{-1}(T_{\operatorname{Hilb}^{d_r}(\A^2),\mu}^{\vee})}=\operatorname{cont}_{\mu}(\wedge_{-1}^\bullet t_1t_2\mathcal{O}^{[d_r]}_{\A_2}).\]
We therefore get
\[\sum_{\mu\vdash d_r}P_{\mu/\varnothing}=\sum_{\mu\vdash d_r}\operatorname{cont}_{\mu}(t_1t_2\mathcal{O}^{[d_r]}_{\A_2})=\chi_T(\operatorname{Hilb}^{d_r}(\A^2),t_1t_2\mathcal{O}^{[d_r]}_{\A_2})\]
as desired.
\end{proof}
\bibliographystyle{alpha}
\bibliography{ref}
\end{document}